\documentclass[11pt,a4paper,twoside,english]{article}
\RequirePackage{babel}
\usepackage[T1]{fontenc}
\usepackage[utf8]{inputenc}
\RequirePackage{microtype}
\RequirePackage{csquotes}

\usepackage[charter,expert,cal=cmcal]{mathdesign}
\usepackage[scaled=.96,sups]{XCharter}
\usepackage[scaled=1.04,varqu,varl]{inconsolata}

\usepackage[onehalfspacing,nodisplayskipstretch]{setspace}

\usepackage[usenames,dvipsnames,table]{xcolor}
	\colorlet{artcolor}{Brown}
	\colorlet{artcolorbg}{white}
	\colorlet{artcolortext}{black}
	\colorlet{notecolor}{Blue}

\colorlet{ARTCOLOR}{artcolor}
\colorlet{ARTCOLORBG}{artcolorbg}
\colorlet{ARTCOLORTEXT}{artcolortext}
\colorlet{NOTECOLOR}{notecolor}
\colorlet{artcolor-secnumber}{artcolor!90!artcolorbg}
\colorlet{artcolor-secnumberbg}{artcolor!12!artcolorbg}
\colorlet{artcolor-secnamebg}{artcolor!5!artcolorbg}
\colorlet{artcolor-abstractbg}{artcolor!3!artcolorbg}
\colorlet{artcolor-boxbg}{artcolor!3!artcolorbg}
	\colorlet{artcolor-boxline}{artcolor-secnumberbg}
\usepackage[pagecolor={artcolorbg}]{pagecolor}
\color{artcolortext}

  \usepackage[nomarginpar,hmargin=2.5cm,vmargin=2.5cm]{geometry}

\PassOptionsToPackage{inline}{enumitem}
\RequirePackage{enumitem}

\usepackage{tabularx,booktabs}

\RequirePackage{mathtools}
\usepackage[framemethod=TikZ]{mdframed}
\PassOptionsToPackage{amsmath,thmmarks,hyperref}{ntheorem}
\RequirePackage{ntheorem}
\RequirePackage{thmtools}
\RequirePackage{fixmath}
\RequirePackage{xfrac}
\RequirePackage{cases}
\RequirePackage{latexsym}
\RequirePackage{upgreek}

\RequirePackage{scalerel}
\RequirePackage{tikz}
\usetikzlibrary{svg.path}
\definecolor{orcidlogocol}{HTML}{A6CE39}
\tikzset{%
  orcidlogo/.pic={%
    \fill[orcidlogocol] svg{M256,128c0,70.7-57.3,128-128,128C57.3,256,0,198.7,0,128C0,57.3,57.3,0,128,0C198.7,0,256,57.3,256,128z};
    \fill[white] svg{M86.3,186.2H70.9V79.1h15.4v48.4V186.2z}
                 svg{M108.9,79.1h41.6c39.6,0,57,28.3,57,53.6c0,27.5-21.5,53.6-56.8,53.6h-41.8V79.1z M124.3,172.4h24.5c34.9,0,42.9-26.5,42.9-39.7c0-21.5-13.7-39.7-43.7-39.7h-23.7V172.4z}
                 svg{M88.7,56.8c0,5.5-4.5,10.1-10.1,10.1c-5.6,0-10.1-4.6-10.1-10.1c0-5.6,4.5-10.1,10.1-10.1C84.2,46.7,88.7,51.3,88.7,56.8z};
  }%
}
\RequirePackage{url}
\DeclareRobustCommand\orcid[1]{\texorpdfstring{\href{https://orcid.org/#1}{\mbox{\scalerel*{%
\begin{tikzpicture}[yscale=-1,transform shape]%
\pic{orcidlogo};%
\end{tikzpicture}%
}{|}}}}{}}

\PassOptionsToPackage{pdfusetitle,hidelinks,bookmarksnumbered,bookmarksopen}{hyperref}
\RequirePackage{hyperref}
\RequirePackage{cleveref}
\crefname{assumption}{assumption}{assumptions}
\crefname{ineq}{inequality}{inequalities}
\creflabelformat{ineq}{\upshape(#2#1#3)}
\crefformat{ineq}{ineq.~\upshape(#2#1#3)}

\theoremstyle{nonumberplain}
\theoremheaderfont{\normalfont\itshape\color{artcolor}}
\theoremseparator{.}
\theorembodyfont{\normalfont}
\theoremsymbol{\small\(\color{artcolor}\square\)}
\newtheorem{proof}{Proof}

\theoremsymbol{}
\declaretheoremstyle[
    headfont=\normalfont\color{artcolor}, 
    bodyfont=\normalfont\itshape\color{artcolortext},
    headpunct={.},
    spaceabove=\parsep,
    mdframed={
      	linecolor=artcolor-boxline,
      	linewidth=1pt,
      	backgroundcolor=artcolor-boxbg,
  	innerleftmargin=5pt,
  	innerrightmargin=6pt,
         innertopmargin=6pt,
         innerbottommargin=6pt,         
  	topline=false,rightline=false,bottomline=false,
  	skipabove=1.2\topsep,
  	skipbelow=0pt} 
]{thmstyle}
\declaretheorem[style=thmstyle, name=Theorem,     numberwithin=section]{theorem}
\declaretheorem[style=thmstyle, name=Proposition, numberlike=theorem]{proposition}
\declaretheorem[style=thmstyle, name=Lemma,       numberlike=theorem]{lemma}
\declaretheorem[style=thmstyle, name=Corollary,     numberlike=theorem]{corollary}
\declaretheoremstyle[
    headfont=\color{artcolor}, 
    bodyfont=\normalfont\color{artcolortext},
    headpunct={.},
    spaceabove=\parsep,
    mdframed={
      	linecolor=artcolor-boxline,
      	linewidth=1pt,           
      	backgroundcolor=artcolor-boxbg,  
  	innerleftmargin=5pt,
  	innerrightmargin=6pt,
         innertopmargin=6pt,
         innerbottommargin=6pt,         
  	topline=false,rightline=false,bottomline=false,
  	skipabove=1.2\topsep,
  	skipbelow=0pt} 
]{defstyle}

\declaretheoremstyle[
    headfont=\color{artcolor}, 
    bodyfont=\normalfont\color{artcolortext},
    headpunct={.},
    spaceabove=\parsep,
    mdframed={
      	linecolor=artcolor-boxline,
      	linewidth=1pt,          
      	backgroundcolor=artcolorbg,  
  	innerleftmargin=5pt,
  	innerrightmargin=6pt,
          innertopmargin=6pt,
          innerbottommargin=6pt,         
  	 topline=false,rightline=false,bottomline=false,
  	 skipabove=1.2\topsep,
  	 skipbelow=0pt} 
]{remstyle}
\declaretheoremstyle[
    headfont=\color{artcolor}, 
    bodyfont=\normalfont\color{artcolortext},
    headpunct={.},
    spaceabove=\parsep,
    mdframed={
      	linecolor=artcolor-boxline,
      	linewidth=1pt,          
      	backgroundcolor=artcolor-boxbg,  
  	innerleftmargin=5pt,
  	innerrightmargin=6pt,
          innertopmargin=6pt,
          innerbottommargin=6pt,         
  	 topline=false,rightline=false,bottomline=false,
  	 skipabove=1.2\topsep,
  	 skipbelow=0pt} 
]{nonumberremstyle}
\declaretheorem[style=remstyle, name=Remark,  numberlike=theorem]{remark}

\RequirePackage{titling}
\makeatletter
\pretitle{\begin{center}\fontsize{.68cm}{.9cm}\selectfont\color{artcolor}\itshape}
\posttitle{\par\end{center}}
\newcommand\@address{}
\newcommand\address[1]{\renewcommand\@address{#1}}
\RequirePackage{url}
\newcommand\@email{}
\newcommand\email[1]{\renewcommand\@email{#1}}
\preauthor{\vspace*{-.8em}\begin{center}\large\color{artcolortext}}
\postauthor{\par{\scriptsize{\color{artcolor}\@email}\par\vskip .3em \textsc{\@address}\par}\end{center}\vspace*{-1em}}
\predate{\begin{center}\color{artcolor}\small\itshape}
\postdate{\par\end{center} \vskip -1em\color{artcolortext}}
\makeatother

\renewenvironment{abstract}{%
\begin{center}
\begin{minipage}{.85\textwidth}
\begin{mdframed}[backgroundcolor=artcolor-abstractbg,hidealllines=true]\small\textsc{\color{artcolor}\normalsize Abstract.}\color{artcolortext}\@}
{\end{mdframed}\end{minipage}\end{center}\color{artcolortext}\vskip\topsep}

\newcommand\subjclass[2][2010]{\vspace{-3\topsep}\begin{center}\begin{minipage}{.85\textwidth}\begin{mdframed}[backgroundcolor=artcolorbg,hidealllines=true]\footnotesize Mathematics Subject Classification (#1): #2.\end{mdframed}\end{minipage}\end{center}\vspace*{-1em}}

\newcommand\keywords[1]{\vspace{-3\topsep}\begin{center}\begin{minipage}{.85\textwidth}\begin{mdframed}[backgroundcolor=artcolorbg,hidealllines=true]\footnotesize\textit{Key words and phrases:} #1.\end{mdframed}\end{minipage}\end{center}}

\RequirePackage{lastpage}
\RequirePackage{fancyhdr}
\makeatletter
\fancypagestyle{plain}{%
\fancyhf{} 
\fancyfoot[C]{\scshape\scriptsize\thepage\ of \pageref{LastPage}} 

}

\newcommand\@titlerunning{}
\newcommand\titlerunning[1]{\renewcommand\@titlerunning{#1}}
\newcommand\@authorrunning{}
\newcommand\authorrunning[1]{\renewcommand\@authorrunning{#1}}

\makeatother

\usepackage[explicit]{titlesec}
\titleformat{\section}{\normalsize}{}{0pt}{%
\begin{tabularx}{\linewidth}{lX}%
\large\cellcolor{artcolor-secnumberbg}\textcolor{artcolor-secnumber}{\thesection} & \large\cellcolor{artcolor-secnamebg}\textcolor{artcolor}{#1}%
\end{tabularx}%
}
\titleformat{name=\section,numberless}{\normalsize}{}{0pt}{%
\begin{tabularx}{\linewidth}{lX}%
\large\cellcolor{artcolor-secnumberbg}\textcolor{artcolor-secnumber}{\Large${\star}$} & \large\cellcolor{artcolor-secnamebg}\textcolor{artcolor}{#1}%
\end{tabularx}%
}
\titleformat{\subsection}[runin]{\normalsize}{}{0pt}{%
\begin{tabular}{ll}%
\cellcolor{artcolor-secnumberbg}\textcolor{artcolor-secnumber}{\thesubsection} & \cellcolor{artcolor-secnamebg}\textcolor{artcolor}{#1.}%
\end{tabular}%
}
\titleformat{name=\subsection,numberless}[runin]{\normalsize}{}{0pt}{%
\begin{tabular}{ll}%
\cellcolor{artcolor-secnumberbg}\textcolor{artcolor-secnumber}{${\star}$} & \cellcolor{artcolor-secnamebg}\textcolor{artcolor}{#1.}%
\end{tabular}%
}

\RequirePackage{titletoc}
\titlecontents{section}
		   [1pc]
		   {\addvspace{2pt}}
		   {\contentsmargin{0pt}\small{\color{artcolor}\contentslabel{1pc}}}
		   {\hspace*{-1pc}\small{\color{artcolor}\parbox{1pc}{\normalsize\({ \star }\)}}}
		   {\small\itshape\contentspage}
		   []	   
\titlecontents{subsection}
		   [2.5pc]
		   {}
		   {\footnotesize\contentslabel{1.5pc}}
		   {\hspace*{-1.5pc}\small{\color{artcolor}\parbox{1.5pc}{\small\({ \star }\)}}}
		   {\footnotesize\itshape\contentspage}
		   []
\RequirePackage{multicol}

\renewcommand\tableofcontents{\vskip\topsep%
    \section*{\contentsname
        \@mkboth{%
           \MakeUppercase\contentsname}{\MakeUppercase\contentsname}}%
           \vspace*{-1em}%
     \begin{multicols}{2}%
    \raggedcolumns\@starttoc{toc}%
    \end{multicols}%
    }

\usepackage[
backend=biber,
bibstyle=ieee,
citestyle=numeric-comp,
sorting=nyt,
url=false,
isbn=false,
doi=false
]{biblatex}
\renewcommand*{\bibfont}{\scriptsize}

\newcommand{\doiorurl}{%
  \iffieldundef{doi}
    {\iffieldundef{url}
       {}
       {\strfield{url}}}
    {https://dx.doi.org/\strfield{doi}}%
}

\newcommand{\myhref}[1]{%
 \ifboolexpr{%
   test {\ifhyperref}
   and
   not test {\iftoggle{bbx:url}}
   and
   not test {\iftoggle{bbx:doi}}
  }
  {\href{\doiorurl}{#1}}
  {#1}%
}

\DeclareFieldFormat{title}{\myhref{\mkbibemph{#1}}}
\DeclareFieldFormat
  [article,inbook,incollection,inproceedings,patent,thesis,unpublished]
  {title}{\myhref{\mkbibquote{#1\isdot}}}

\usepackage{wrapfig}
\usepackage[font=normalsize,labelfont={color=artcolor}]{caption}
\usepackage{subcaption}
\usepackage{tikz}
\usepackage{pgfplots}
\pgfplotsset{compat=newest}
\usetikzlibrary{fit,arrows.meta,shapes}

\makeatletter
\renewcommand*\env@matrix[1][c]{\hskip -\arraycolsep
  \let\@ifnextchar\new@ifnextchar
  \array{*\c@MaxMatrixCols #1}}
\makeatother

\makeatletter
\def\hdots@for#1#2{\multicolumn{#2}c%
  {\m@th\dotsspace@1.5mu\mkern-#1\dotsspace@
   \xleaders\hbox{$\m@th\mkern#1\dotsspace@\cdot\mkern#1\dotsspace@$}%
           \hskip\z@\@plus 1filll
   \mkern-#1\dotsspace@}%
   }
\makeatother

\RequirePackage{etoolbox}
\makeatletter
\patchcmd{\@adjustvertspacing}
  {\jot\baselineskip \divide\jot 4}
  {\jot=.5\baselineskip}
  {}{}
\makeatother

\newcommand*{\diffdchar}{\textnormal{d}}
\newcommand*{\dee}{\mathop{\diffdchar\!}}

\DeclareMathOperator{\support}{supp}

\DeclareMathOperator{\sign}{sgn}

\DeclarePairedDelimiterX{\norm}[1]{\lVert}{\rVert}{\ifblank{#1}{\,\cdot\,}{#1}}
\DeclarePairedDelimiterX{\seminorm}[1]{\lvert}{\rvert}{\ifblank{#1}{\,\cdot\,}{#1}}
\DeclarePairedDelimiterX{\abs}[1]{\lvert}{\rvert}{\ifblank{#1}{\,\cdot\,}{#1}}
\DeclarePairedDelimiterX{\innerproduct}[2]{\langle}{\rangle}{\ifblank{#1}{\,\cdot\,}{#1},\ifblank{#2}{\,\cdot\,}{#2}}
\DeclarePairedDelimiterX{\dualitypairing}[2]{\langle}{\rangle}{\ifblank{#1}{\,\cdot\,}{#1},\ifblank{#2}{\,\cdot\,}{#2}}
\DeclarePairedDelimiterXPP{\distance}[2]{\operatorname{dist}}{(}{)}{}{\ifblank{#1}{\,\cdot\,}{#1},\ifblank{#2}{\,\cdot\,}{#2}}

\providecommand\given{}
 \newcommand\SetSymbol[1][]{%
      \nonscript\: :
      \allowbreak
      \nonscript\:
      \mathopen{}}
   \DeclarePairedDelimiterX\Set[1]\{\}{%
      \renewcommand\given{\SetSymbol[\delimsize]}
      #1
   }

\makeatletter
\newcommand{\vast}{\bBigg@{3}}
\newcommand{\Vast}{\bBigg@{5}}
\makeatother

\renewcommand\leq\leqslant
\renewcommand\geq\geqslant

\newcommand{\smalloh}{o}

\DeclareRobustCommand{\A}[0]{\mathbb{A}}
\DeclareRobustCommand{\B}[0]{\mathbb{B}}
\DeclareRobustCommand{\C}[0]{\mathbb{C}}
\DeclareRobustCommand{\D}[0]{\mathbb{D}}
\DeclareRobustCommand{\E}[0]{\mathbb{E}}
\DeclareRobustCommand{\F}[0]{\mathbb{F}}
\DeclareRobustCommand{\G}[0]{\mathbb{G}}
\DeclareRobustCommand{\H}[0]{\mathbb{H}}
\DeclareRobustCommand{\I}[0]{\mathbb{I}}
\DeclareRobustCommand{\J}[0]{\mathbb{J}}
\DeclareRobustCommand{\K}[0]{\mathbb{K}}
\DeclareRobustCommand{\L}[0]{\mathbb{L}}
\DeclareRobustCommand{\M}[0]{\mathbb{M}}
\DeclareRobustCommand{\N}[0]{\mathbb{N}}
\DeclareRobustCommand{\O}[0]{\mathbb{O}}
\DeclareRobustCommand{\P}[0]{\mathbb{P}}
\DeclareRobustCommand{\Q}[0]{\mathbb{Q}}
\DeclareRobustCommand{\R}[0]{\mathbb{R}}
\DeclareRobustCommand{\S}[0]{\mathbb{S}}
\DeclareRobustCommand{\T}[0]{\mathbb{T}}
\DeclareRobustCommand{\U}[0]{\mathbb{U}}
\DeclareRobustCommand{\V}[0]{\mathbb{V}}
\DeclareRobustCommand{\W}[0]{\mathbb{W}}
\DeclareRobustCommand{\X}[0]{\mathbb{X}}
\DeclareRobustCommand{\Y}[0]{\mathbb{Y}}
\DeclareRobustCommand{\Z}[0]{\mathbb{Z}}

\DeclareRobustCommand{\calE}[0]{{\mathcal E}}

\DeclareRobustCommand{\calL}[0]{{\mathcal L}}

\DeclareRobustCommand{\calN}[0]{{\mathcal N}}
\DeclareRobustCommand{\calO}[0]{{\mathcal O}}

\DeclareRobustCommand{\calQ}[0]{{\mathcal Q}}

\DeclareRobustCommand{\calU}[0]{{\mathcal U}}

\DeclareRobustCommand{\scrF}[0]{{\mathscr F}}

\makeatletter
\let\save@mathaccent\mathaccent
\newcommand*\if@single[3]{%
  \setbox0\hbox{${\mathaccent"0362{#1}}^H$}%
  \setbox2\hbox{${\mathaccent"0362{\kern0pt#1}}^H$}%
  \ifdim\ht0=\ht2 #3\else #2\fi
  }
\newcommand*\rel@kern[1]{\kern#1\dimexpr\macc@kerna}
\newcommand*\widebar[1]{\@ifnextchar^{{\wide@bar{#1}{0}}}{\wide@bar{#1}{1}}}
\newcommand*\wide@bar[2]{\if@single{#1}{\wide@bar@{#1}{#2}{1}}{\wide@bar@{#1}{#2}{2}}}
\newcommand*\wide@bar@[3]{%
  \begingroup
  \def\mathaccent##1##2{%
    \let\mathaccent\save@mathaccent
    \if#32 \let\macc@nucleus\first@char \fi
    \setbox\z@\hbox{$\macc@style{\macc@nucleus}_{}$}%
    \setbox\tw@\hbox{$\macc@style{\macc@nucleus}{}_{}$}%
    \dimen@\wd\tw@
    \advance\dimen@-\wd\z@
    \divide\dimen@ 3
    \@tempdima\wd\tw@
    \advance\@tempdima-\scriptspace
    \divide\@tempdima 10
    \advance\dimen@-\@tempdima
    \ifdim\dimen@>\z@ \dimen@0pt\fi
    \rel@kern{0.6}\kern-\dimen@
    \if#31
      \overline{\rel@kern{-0.6}\kern\dimen@\macc@nucleus\rel@kern{0.4}\kern\dimen@}%
      \advance\dimen@0.4\dimexpr\macc@kerna
      \let\final@kern#2%
      \ifdim\dimen@<\z@ \let\final@kern1\fi
      \if\final@kern1 \kern-\dimen@\fi
    \else
      \overline{\rel@kern{-0.6}\kern\dimen@#1}%
    \fi
  }%
  \macc@depth\@ne
  \let\math@bgroup\@empty \let\math@egroup\macc@set@skewchar
  \mathsurround\z@ \frozen@everymath{\mathgroup\macc@group\relax}%
  \macc@set@skewchar\relax
  \let\mathaccentV\macc@nested@a
  \if#31
    \macc@nested@a\relax111{#1}%
  \else
    \def\gobble@till@marker##1\endmarker{}%
    \futurelet\first@char\gobble@till@marker#1\endmarker
    \ifcat\noexpand\first@char A\else
      \def\first@char{}%
    \fi
    \macc@nested@a\relax111{\first@char}%
  \fi
  \endgroup
}
\makeatother
\color{artcolortext} 

\pgfdeclareradialshading[fradialcolour]{fradial}{\pgfpoint{0}{0}}{%
  color(0)=(artcolorbg);%
  color(15bp)=(artcolorbg!80!artcolor-secnumberbg);%
  color(30bp)=(artcolor!10!artcolor-secnumberbg)%
}
\colorlet{fradialcolour}{artcolor-secnamebg}

\def\Vhrulefill{\leavevmode\leaders\hrule height 0.7ex depth \dimexpr0.4pt-0.7ex\hfill\kern0pt}

\begin{document}
\title{Solitary waves in dispersive evolution equations\\ of Whitham type with nonlinearities of mild regularity}
\titlerunning{Solitary waves in dispersive equations of Whitham type}
\author{\protect\phantom{\,\orcid{0000-0002-9905-1670}}Fredrik Hildrum\,\protect\orcid{0000-0002-9905-1670}} 
\authorrunning{Fredrik Hildrum}
\address{Department of Mathematical Sciences,\\ NTNU -- Norwegian University of Science and Technology,\\ 7491 Trondheim, Norway}
\email{fredrik.hildrum@ntnu.no}
\date{\today}
\renewcommand\footnotemark{}
\renewcommand\footnoterule{}
\maketitle

\begin{abstract}
We show existence of small solitary and periodic traveling-wave solutions in Sobolev spaces~\({\textnormal{H}^s}\), \({ s > 0 }\), to a class of nonlinear, dispersive evolution equations of the form
\begin{equation*}
u_t + \left(Lu+ n(u)\right)_x = 0,
\end{equation*}
where the dispersion~\({L}\) is a negative-order Fourier multiplier whose symbol is of KdV type at low frequencies and has integrable Fourier inverse~\({ K }\) and the nonlinearity~\({n}\) is inhomogeneous, locally Lipschitz and of superlinear growth at the origin. This generalises earlier work by Ehrnström, Groves~\& Wahlén on a class of equations which includes Whitham's model equation for surface gravity water waves featuring the exact linear dispersion relation. Tools involve constrained variational methods, Lions' concentration-compactness principle, a~strong fractional chain rule for composition operators of low relative regularity, and a~cut-off argument for~\({n}\) which enables us to go below the typical~\({s > \frac{1}{2}}\)~regime. We~also demonstrate that these solutions are either waves of elevation or waves of depression when \({ K }\) is nonnegative, and provide a nonexistence result when~\({ n }\) is too~strong.
\end{abstract}


\section{Introduction}

\subsection{Background}

Many model equations for one-dimensional spacial evolution of water waves~\autocite{Lannes2013} may be written~as
\begin{equation}\label{eq:evolution-equation}
u_t + (Lu + n(u))_x = 0,
\end{equation}
where~\({ L }\) is a dispersive Fourier multiplier operator in space and~\({ n }\) represents local nonlinear effects. Much effort has been put into answering whether~\eqref{eq:evolution-equation} admits traveling-wave solutions---and in particular, \emph{solitary~waves}. Propagating with fixed speed~\({\!\nu}\) and shape, these solutions take the form~\({ (t, x) \mapsto u(x - \nu t) }\) with \({ u(y) \to 0 }\) as~\({ \abs{ y } \to \infty }\), and satisfy
\begin{equation}\label{eq:traveling-wave}
Lu - \nu u + n(u) = 0
\end{equation}
after integrating~\eqref{eq:evolution-equation}.

\begin{wrapfigure}{r}{.38\textwidth}%
\centering%
\includegraphics{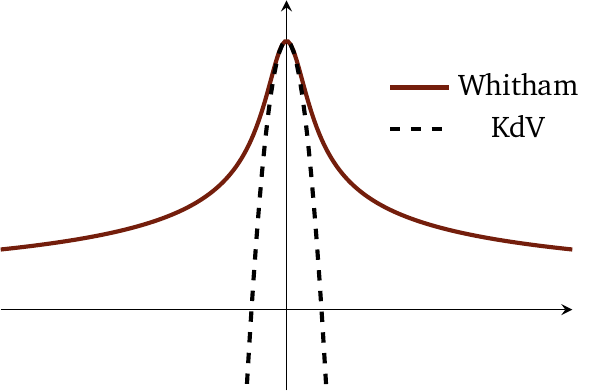}
\caption{Whitham and KdV symbols.}%
\label{fig:symbols}%
\vspace{10pt}%
\end{wrapfigure}
In 1967 Whitham~\autocite{Whitham1967,Whitham1974} proposed a shallow-water model of type~\eqref{eq:evolution-equation} with~\({ n(u) = u^2 }\) and
\begin{equation*}
\mathscr{F}( Lu) (\xi) = \sqrt{\frac{\tanh \xi}{\xi}} \; \widehat{  u  }(\xi)
\end{equation*}
as an alternative to the Korteweg--de~Vries (KdV) equation featuring the exact linear dispersion relation for unidirectional water waves influenced by gravity. As~seen from
\begin{equation*}
\mathfrak{m}(\xi) \coloneqq \sqrt{\frac{\tanh \xi}{\xi}} = \underbrace{1 - \tfrac{1}{6} \xi^2}_{ \textnormal{KdV symbol}} + \mathcal{O}\bigl(\xi^4\bigr)
\end{equation*}
and \cref{fig:symbols}, it is intuitively reasonable that Whitham's model should both perform better and on  a wider range of wave numbers than the KdV~equation.

Unfortunately, the nonlocal, singular nature of~\({ L }\)---due to~\({ \mathfrak{m}(\xi) \lesssim \langle \xi \rangle^{- \frac{1}{2}} }\) being inhomogeneous and decaying very slowly at infinity---seems to have prevented people from rigorously studying the Whitham equation until recently. Significant breakthrough in the last decade, however, has put the original Whitham equation, and also other full-dispersion models, in the spotlight, beginning with the existence of periodic traveling waves by Ehrnström and Kalisch~\autocite{EhrnstromKalisch2009} in~2009 and solitary-wave solutions by Ehrnström, Groves and Wahlén~\autocite{EGW2012} in~2012; see~also~\autocite{SteWri2018a}. Research has furthermore confirmed Whitham's conjectures for qualitative wave breaking (bounded wave profile with unbounded slope) in finite time~\autocite{Hur2017a} and the existence of highest, cusp-like solutions~\autocite{EhrWah2019a,EhrnstromKalisch2013}---now known to also have a convex profile between the stagnation points~\autocite{EncGomVer2018a}.

Additional analytical and numerical results for the Whitham equation include modulational instability of periodic waves~\autocite{Sanfordetal2014, HurJohnson2015a}, local well-posedness in Sobolev spaces~\({ \textnormal{H}^{ s} }\), \({ s > \frac{ 3 }{ 2 } }\), for both solitary and periodic initial data~\autocite{EhrnstromEscherPei2015,KleLinPilSau2018a,EhrPei2018a}, non-uniform continuity of the data-to-solution map~\autocite{Arn2019a}, symmetry and decay of traveling waves~\autocite{BruEhrPei2017a}, analysis of modeling properties, dynamics and identification of scaling regimes~\autocite{KleLinPilSau2018a}, and wave-channel experiments and other numerical studies~\autocite{BorlukKalischNicholls2013,TriKleClaOno2016a,KalMolVer2017a,Car2018a}. 

In~total, these investigations have demonstrated the potential usefulness of full-dispersion versions of traditional shallow-water models.

\subsection{Assumptions and main results}

In this paper we contribute to the longstanding mathematical program of fully understanding the interplay between dispersive and nonlinear effects for the formation of traveling waves. Specifically, we generalise~\autocite{EGW2012}, in which the authors proved the existence of small solitary and periodic traveling-wave solutions in the Sobolev space~\({ \textnormal{H}^1 }\) to a family of equations of the form~\eqref{eq:evolution-equation} with \enquote{Whitham-type} symbols---that is, negative-order, inhomogeneous symbols~\({ \mathfrak{m} }\) with KdV-type behaviour at low frequencies---and inhomogeneous nonlinearities~\({ n }\) being at least quadratic near the~origin. Under the following assumptions, we study the existence of solutions to~\eqref{eq:traveling-wave} in fractional Sobolev spaces both on the real line and in the periodic setting, noting that \({ \sigma = - \frac{1}{2} }\), \({ \ell = 1 }\) and~\({ q = 1 }\) for the original Whitham equation.

\clearpage 

\begin{enumerate}[leftmargin=*,label=\textcolor{artcolor}{\normalfont\({\textnormal{A}_{\arabic*}}\):},ref={\({\textnormal{A}_{\arabic*}}\)}]\itshape
\item \label{assumption:dispersion} \textcolor{artcolor}{\textit{Linear, nonlocal dispersive term.}\quad\Vhrulefill}
\begin{enumerate}[leftmargin=0pt,label=\textcolor{artcolor}{\normalfont\roman*)},ref={\roman*)},topsep=0pt]
\item
\({L}\) is a Fourier multiplier operator with even, inhomogeneous symbol~\({\mathfrak{m} \colon \R \to \R}\) of order~\({ \sigma < 0 }\), that~is,
\begin{equation*}
\widehat{Lu} = \mathfrak{m}\, \widehat{u} \qquad \text{and} \qquad \abs{\mathfrak{m}(\xi)} \lesssim \langle \xi \rangle^{\sigma},
\end{equation*}
where~\({\langle \xi \rangle \coloneqq \sqrt{1 + \xi^2}}\).
\item
\({\mathfrak{m} }\) is in the Wiener class~\({ \textnormal{W}_0 }\) of functions with absolutely integrable inverse Fourier transform, so~that~\({ L }\) is a convolution operator
\begin{equation*}
Lu = \tfrac{ 1 }{ \sqrt{2 \uppi} } K \ast u
\end{equation*}
with kernel~\({ K \coloneqq \mathscr{F}^{-1}(\mathfrak{m}) \in \textnormal{L}^1 }\).
\item \label{assumption:dispersion-expansion}
\({\mathfrak{m}}\) has a strictly positive unique global maximum at~\({0}\) and is \({ \textnormal{C}^{2\ell} }\)-regular around~\({ 0 }\) for some~\({\ell \in \mathbb{Z}_+}\), with~\({\mathfrak{m}^{(2\ell)}(0) < 0}\). Thus~\({ \mathfrak{m} }\) has the Maclaurin expansion
\begin{equation*}
\mathfrak{m}(\xi) = \mathfrak{m}(0) + \frac{\mathfrak{m}^{(2\ell)}(0)}{(2\ell)!} \xi^{2\ell} + \mathcal{O}\left( \abs{\xi}^{2\ell+ 2} \right).
\end{equation*}
\end{enumerate}
\item \label{assumption:nonlinearity} \textcolor{artcolor}{\textit{Nonlinearity.}\quad\Vhrulefill} \\
\({n \colon \R \to \R}\) is locally Lipschitz continuous (\({ n \in \textnormal{Lip}_{\textnormal{loc}} }\)) and of the form
\begin{equation*}
n(x) = n_q(x) +  n_\textnormal{r}(x),
\end{equation*}
where the leading-order term, with~\({q \in (0, 4\ell)}\), equals
\begin{equation*}
n_q(x) = \gamma |x|^{1 + q} \qquad \text{or} \qquad n_q(x) = \gamma x|x|^{q}
\end{equation*}
for a constant~\({\gamma \neq 0}\) or~\({\gamma > 0}\), respectively, and the remainder satisfies
\begin{equation*}
n_\textnormal{r}^{(j)}(x) = o \bigl( \abs{x}^{1 + q - j} \bigr)
\end{equation*}
as~\({x \to 0}\) for all~\({ j = 0, 1, \dotsc, \lfloor \varsigma \rfloor }\) if~\({ n \in  \textnormal{C}_{ \textnormal{loc}}^{  \varsigma}  }\) for some real~\({ \varsigma < 1 + q }\). In particular,
\begin{equation*}
n^{(j)}(x) = \calO \left(\abs{x}^{1 + q - j}\right) \qquad \text{for all } j = 0, \dotsc, \lfloor \varsigma \rfloor.
\end{equation*}
When~\({ n }\) is just in \({ \textnormal{Lip}_{\textnormal{loc}} }\), we assume that~\({ n'(x) = \mathcal{O}(\abs{ x }^q) }\) almost everywhere as~\({ x \to 0 }\).
\end{enumerate}

\begin{remark}
We write \({ A \lesssim B }\) or \({ B \gtrsim A }\) if~\({ A \leq c B }\) for some constant~\({ c > 0 }\) independent of~\({ A }\) and~\({ B }\), and~\({ A \eqsim B }\) symbolises that~\({ A \lesssim B \lesssim A }\).
\end{remark}

\enlargethispage{\baselineskip} 

In comparison to~\autocite{EGW2012} we consider more general symbols and nonlinearities. We allow for nonlinearities that are merely locally Lipschitz continuous and of superlinear growth (\({ q > 0 }\)) at the origin, down from~\({ n \in \textnormal{C}^2 }\) with at least quadratic growth (\({ q \geq 1 }\)) in~\autocite{EGW2012}. In order to allow~\({ q \in (0, 1) }\), we on the one hand make use of an order-optimal fractional chain rule; see~\eqref{eq:chain-rule-intro} and \cref{sec:nonlinearity}. On the other hand, we invoke, among other, the Gagliardo--Nirenberg inequality at a certain step, see \cref{sec:variational-method,sec:subadditivity}, which both improves upon and simplifies the corresponding estimates in~\autocite{EGW2012}. The upper bound \({ q < 4\ell }\), however, is the same in both articles, and we establish that this bound is, in fact, optimal for small solitary waves with sufficiently high speed. Notice also in Assumption~\ref{assumption:nonlinearity} that there is some decoupling of the regularity and the growth of~\({ n }\) in the sense that~\({ \varsigma < 1 + q }\).

As regards the dispersive term, the KdV-type behaviour of~\({ \mathfrak{m} }\) at low frequencies in Assumption~\ref{assumption:dispersion}~\ref{assumption:dispersion-expansion} coincides with that of~\autocite{EGW2012}. When it comes to global regularity and decay, the authors of~\autocite{EGW2012} assumed negative-order symbols~\({ \mathfrak{m}\in \textnormal{S}_{\infty}^{\sigma} }\), that~is, \({ \mathfrak{m} \in \textnormal{C}^\infty }\) and~\({ \abs{ \mathfrak{m}^{(j)}(\xi) } \lesssim \langle \xi \rangle^{\sigma - j} }\) for all~\({ j \in \N_0 }\). This not only implies that~\({ \mathfrak{m} \in \textnormal{W}_0 }\), but also that the kernel~\({ K }\) is essentially very localised, which was used in~\autocite{EGW2012} to control the nonlocal estimates. As an improvement, we show that all of these estimates, in fact, follow from general properties of convolution with an \({ \textnormal{L}^1 }\)~kernel, together with decay on~\({ \mathfrak{m} }\) itself---omitting any assumptions on its derivatives; see \cref{sec:variational-method,sec:action-L,sec:special-min-seq,sec:concentration-compactness} for more details. For convenience, we include in~\Cref{app:wienerclass-sufficientconditions} a list of recent and practical sufficient conditions for symbols to be in~\({ \textnormal{W}_0 }\).

Under Assumptions~\ref{assumption:dispersion} and~\ref{assumption:nonlinearity}, we study~\eqref{eq:traveling-wave} in the Sobolev space~\({ \textnormal{H}^s }\) on the real line and in the corresponding \({ P }\)-periodic analogue~\({ \textnormal{H}_P^s }\) in the periodic setting (see~\cref{sec:preliminaries-spaces} for definitions) for~\({ s > 0 }\) satisfying
\begin{equation} \label{eq:sobolev-index}
\tfrac{ 1 }{ 2 } - \abs{ \sigma } < s < \varsigma, \qquad \text{with } \varsigma < 1 + q,
\end{equation}
and obtain the following main results.
\begin{theorem}[Periodic traveling waves] \label{thm:existence-periodic}
For each sufficiently small~\({ \mu > 0 }\) there exists a period~\({ P_\mu > 0 }\), such that for all~\({ P \geq P_\mu }\) equation~\eqref{eq:traveling-wave} admits a nonconstant solution \({ u \in \textnormal{H}_P^s \cap \textnormal{L}^{ \infty } }\) with~\({ \norm{ u }_{ \textnormal{L}_P^2 }^{ 2 } = 2 \mu }\) and supercritical wave speed~\({\!\nu_P > \mathfrak{m}(0) }\). Uniformly over~\({ P \geq P_\mu }\) these solutions satisfy
\begin{gather*}
\nu_P - \mathfrak{m}(0) \eqsim \mu^{q \alpha} \eqsim \norm{ u }_{ \infty }^{ q } , \\
\shortintertext{where~\({ \alpha \coloneqq \frac{2 \ell}{4 \ell - q} > \frac{1}{2}}\), and}
 \norm{ u }_{ \textnormal{H}_P^s } \eqsim \mu^{ \frac{1}{2}}.
 \end{gather*}
\end{theorem}

\begin{theorem}[Solitary waves] \label{thm:existence-solitary}
For each sufficiently small~\({ \mu > 0 }\) there exists a solution~\({ u \in \textnormal{H}^s \cap \textnormal{L}^{ \infty} }\) to~\eqref{eq:traveling-wave} with supercritical speed~\({\!\nu > \mathfrak{m}(0) }\) and~\({ \norm{ u }_{ 0 }^{ 2 } = 2 \mu }\) satisfying
\begin{gather*}
\nu - \mathfrak{m}(0) \eqsim \mu^{q \alpha} \eqsim \norm{ u }_{ \infty }^{ q },  \\
\shortintertext{where~\({ \alpha}\) is as in~\Cref{thm:existence-periodic}, and}
\norm{ u }_{ s } \eqsim \mu^{ \frac{1}{2}}.
\end{gather*}
\end{theorem}
\begin{remark}
\Cref{thm:existence-periodic,thm:existence-solitary} also hold
\begin{enumerate}[itemsep=0pt,topsep=4pt]
\item
with no upper bound on~\({ s }\) if~\({ n }\) is a polynomial with least-order term of~order~\({ 1 + q \in \Z_+ }\);
\item
for \({ s = 1 }\) when~\({ n }\) is just Lipschitz or \({ \textnormal{C}^1 }\) around the origin.
\end{enumerate}
Even if~\({ n'(x) = \mathcal{O}(\abs{ x }^q) }\) a.e.\@ as \({ x \to 0 }\) does not hold in the \({ \textnormal{Lip}_{\textnormal{loc}} }\) case, we still obtain solutions \({u \in \textnormal{H}_P^s \cap \textnormal{L}^{\infty}}\) satisfying, uniformly over \({ P \geq P_\mu }\), the estimates
\begin{equation*}
 \nu_P - \mathfrak{m}(0) \eqsim \mu^{q /2} \qquad \text{and} \qquad \norm{ u }_{ \textnormal{H}_P^s } \eqsim \mu^{ \frac{ 1 }{ 2 }} \eqsim \norm{ u }_{ \infty }.
\end{equation*}
\end{remark}

The \({ \mu }\)-dependent estimates on the wave speed and~\({ \norm{ u }_{ \infty } }\) in \Cref{thm:existence-periodic,thm:existence-solitary} involve the parameter~\({ \alpha }\), which represents a balance between dispersive and nonlinear effects. Since~\({ \alpha = \infty }\) when~\({ q = 4\ell }\), one might expect that there are no nontrivial small solutions of~\eqref{eq:traveling-wave} with speeds close to~\({ \mathfrak{m}(0) }\) if~\({ q \geq 4\ell }\). This is indeed the case in the solitary-wave setting, and is included in \Cref{thm:nonexistence}. 

We also demonstrate in \Cref{thm:wave-sign} that bounded solutions of~\eqref{eq:traveling-wave} with supercritical speed are either waves of elevation or waves of depression in the special case when~\({ K }\) is nonnegative, noting that this result is already known for the Whitham equation~\autocite[Corollary~4.4]{EhrWah2019a}.

In working in fractional Sobolev spaces, both low- and high-order~\({ s }\) come with technical difficulties. As in~\autocite{EGW2012}, we shall treat solutions of~\eqref{eq:traveling-wave} as minimisers of a constrained variational problem, explained in details in~\cref{sec:variational-method}. When~\({ s \leq \frac{1}{2} }\), neither~\({ \textnormal{H}^s }\) nor~\({ \textnormal{H}_P^s }\) are embedded in~\({ \textnormal{L}^\infty }\), which unfortunately means that the minimisation problem is unbounded---even locally. We resolve this issue by a cut-off argument for~\({ n }\) together with the lower bound~\({ s > \frac{ 1 }{ 2 } - \abs{ \sigma } }\) in~\eqref{eq:sobolev-index}. This implies that both~\({ n(u) }\) and~\({ Lu }\) are in~\({ \textnormal{L}^\infty }\), and we have therefore essentially regained~\({ \textnormal{L}^\infty }\) control of~\eqref{eq:traveling-wave}.

Furthermore, we rely on the highly precise fractional chain rule
\begin{equation} \label{eq:chain-rule-intro}
\norm{ n(u) }_{ s } \lesssim \norm{ u }_{ \infty }^q \norm{ u }_{ s }
\end{equation}
on~\({ \textnormal{H}^s \cap \textnormal{L}^\infty }\) by Runst and Sickel~\autocite[Theorem~5.3.4/1~(i)]{RunstSickel1996}, which allows~\({ s }\) to be \emph{arbitrarily} close to~\({ \varsigma }\), and does not seem to be well known. Apart from the immediate case~\({ s \leq 1 }\), an elementary but tedious calculation using the classical higher-order chain rule (Faà di~Bruno's formula) establishes~\eqref{eq:chain-rule-intro} provided~\({ u^{( \lfloor \varsigma \rfloor)} \in \textnormal{L}^\infty }\), that is, when~\({ s > \lfloor \varsigma \rfloor + \frac{ 1 }{ 2 }  }\). The general (high-order) result in~\autocite{RunstSickel1996}, however, is based on technical harmonic analysis.

\subsection{Outline of the variational method} \label{sec:variational-method}

We follow the variational approach in~\autocite{GroWah2011a,EGW2012}, treating solitary-wave solutions as local minimisers of the functional
\begin{equation*}
\mathcal{E}(u) \coloneqq \underbrace{-\frac{1}{2} \int_{\R} u Lu \dee x}_{\displaystyle\eqqcolon \mathcal{L}(u)} \underbrace{- \int_{\R} N(u) \dee x}_{\displaystyle\eqqcolon \mathcal{N}(u)},
\end{equation*}
subject to the constraint that~\({ \mathcal{Q}(u) \coloneqq \frac{1}{2} \int_{\R} u^2 \dee x }\) is held fixed, where
\begin{equation*}
N(x) \coloneqq N_q(x) + N_\textnormal{r}(x), \qquad N_q(x) \coloneqq \frac{x n_q(x)}{2 + q} \qquad \text{and} \qquad  N_\textnormal{r}(x) \coloneqq \int_{0}^{x} n_\textnormal{r}(s) \dee s
\end{equation*}
are primitives of~\({n}\),~\({n_q}\) and~\({n_\textnormal{r}}\) vanishing at~\({0}\). By Lagrange's multiplier principle, any such minimiser~\({ u }\) satisfies
\begin{equation} \label{eq:variational-lagrange-principle}
\mathcal{E}'(u) + \nu \mathcal{Q}'(u) = 0
\end{equation}
for some multiplier~\({\!\nu \in \R}\), which implies that~\({u}\) solves~\eqref{eq:traveling-wave} with wave speed~\({\!\nu}\). Here primes mean representatives of Fréchet derivatives in~\({ \textnormal{L}^2 }\); see~\cref{sec:functionals}.

Specifically, we minimise~\({ \mathcal{E} }\) over a \enquote{constrained ball}
\begin{equation*}
U_\mu^s \coloneqq  \Set*{ u \in \textnormal{H}^s \given \norm{ u }_{ s } < R \text{ and } \mathcal{Q}(u) = \mu }\end{equation*}
for small~\({ \mu, R > 0 }\), and show in~\cref{sec:concentration-compactness} that any minimising sequence which stays away from the \enquote{boundary} \({ \norm{  }_{ s } = R }\) converges---up to subsequences and translations---in~\({ \textnormal{H}^{s-} }\) to a nontrivial solution of~\eqref{eq:traveling-wave} in~\({ \textnormal{H}^s }\) with help of Lions' concentration-compactness principle~\autocite{Lio1984a} adapted to the fractional setting~\autocite[Corollary~3.2]{ParSal2019a}.

\enlargethispage{\baselineskip} 

One must of course confirm the existence of such a minimising sequence. Here the periodic traveling waves come into play. In~\cref{sec:periodic} we consider the corresponding variational problem for \({P}\)-periodic traveling waves with functionals~\({\mathcal{E}_P}\),~\({\mathcal{L}_P}\),~\({\mathcal{N}_P}\) and~\({\mathcal{Q}_P}\), where the domain of integration now is~\({\bigl(-\tfrac{P}{2}, \tfrac{P}{2}\bigr)}\). Both constructively and due to lack of coercivity, we penalise~\({ \mathcal{E}_P }\) so that minimising sequences do not come close to the \enquote{boundary} in~\({ \textnormal{H}_P^s }\). The (generalised) extreme value theorem yields solutions to the penalised problem, and \textit{a~priori} estimates show that the minimisers are unaffected by the penalisation. This establishes most of~\Cref{thm:existence-periodic}, with \emph{uniform} estimates in large~\({ P }\).

We next essentially show that
\begin{equation*}
\left\{\begin{aligned}
\text{the \({ P }\)-periodic traveling-wave problem}\hspace*{1em} \\
\text{scaled, truncated and translated to~\({ \bigl( - \tfrac{ P }{ 2 }, \tfrac{ P }{ 2 } \bigr) }\)}
\end{aligned}\right\}
\quad \xrightarrow[P \to \infty]{} \quad \text{the solitary-wave problem,}
\end{equation*}
and construct a \enquote{boundary-distant} special minimising sequence for the latter with help of the periodic minimisers. Our approach simplifies and extends~\autocite[lemma~3.3 and~theorem~3.8]{EGW2012} in that we only use that~\({ L }\) is a convolution operator with integrable kernel~\({ K }\) in order to deal with the nonlocal effects. In particular, we
neither need to assume algebraic-type decay of~\({ Lu }\) outside~\({ \bigl( - \frac{P}{2}, \frac{P}{2} \bigr)  }\) for~\({ u \in \textnormal{L}^2 }\) supported in~\({ \bigl( - \frac{P}{2}, \frac{P}{2} \bigr)  }\) (see~\autocite[proposition~2.1~(ii)]{EGW2012}), nor that~\({ L }\) commutes with \enquote{the periodisation map}~\autocite[proposition~2.5]{EGW2012}, although we note that this property remains true in our case. As a byproduct, we can also be less restrictive in the truncation process, as long as we have asymptotic control when~\({ P \to \infty }\).

This special minimising sequence, \({ \Set{ \widetilde{u}_k }_k }\), also guarantees that the quantity
\begin{equation*}
I_\mu \coloneqq \inf \Set*{ \mathcal{E}(u) \given u \in U_\mu^s }
\end{equation*}
is \emph{strictly subadditive}, meaning that
\begin{equation} \label{eq:strictly-subadditive}
I_{ \mu_1 + \mu_2} < I_{ \mu_1} + I_{ \mu_2} \qquad \text{whenever } 0 < \mu_1, \mu_2 < \mu_1 + \mu_2 < \mu_\star
\end{equation}
for some~\({ \mu_\star > 0 }\), and is proved in~\cref{sec:subadditivity}. For inhomogeneous~\({ n }\), this relies upon \textit{a~priori} estimates for the size and wave speed of~\({ \widetilde{u}_k }\). Whereas~\autocite{EGW2012} decomposes~\({ \widetilde{u}_k }\) into low- and high-frequency components using sharp frequency cut-offs, we instead apply a smooth decomposition. This seems to be necessary for the estimates to work when~\({ s \leq \frac{ 1 }{ 2 } }\) in order to guarantee that the~\({ \textnormal{L}^\infty }\) norm of the high-frequency component is almost bounded by its \({ \textnormal{H}^s }\)~norm. Furthermore, in order to conclude the \textit{a~priori} estimates, the approach in~\autocite{EGW2012} introduces some scaled Sobolev norms with weights depending on~\({ \mu }\). The arguments~\autocite[proof of Theorem~4.4]{EGW2012} seem to require~\({ q \geq 1 }\), but with help of the Gagliardo--Nirenberg inequality, we found that~\({ q > 0 }\) is possible; see specifically the proof of~\Cref{thm:bound-Hs-norm-mu}.

Strict subadditivity also excludes the unwanted case of dichotomy in Lions'~principle, where we again improve upon~\autocite{EGW2012} by only taking into account that \({ L }\) is a convolution operator. Finally, \textit{a~priori} estimates for the size and speed of traveling waves then complete the proof of~\Cref{thm:existence-periodic,thm:existence-solitary}.

\section{Functional-analytic preliminaries} \label{sec:preliminaries}%

\subsection{Spaces} \label{sec:preliminaries-spaces}

Let
\begin{equation*}
\widehat{\varphi}(\xi) \coloneqq \mathscr{F}(\varphi)(\xi) \coloneqq \frac{1}{\sqrt{2 \uppi}} \int_{\R} \varphi(x) \, \textnormal{e}^{-\textnormal{i}\xi x} \dee x
\end{equation*}
denote the unitary Fourier transform defined initially on the Schwartz space~\({\mathscr{S}}\) and extended by duality to tempered distributions~\({\mathscr{S}'}\). Define~\({\textnormal{L}^q}\), for~\({q \geq 1}\), to be the space of real-valued functions on~\({\R}\) whose norm~\({\norm{u}_{\textnormal{L}^q} \coloneqq \bigl( \int_{\R} \abs{u}^q \dee x \bigr)^{1/q}}\) is finite, with~\({\norm{u}_{\infty} \coloneqq \operatorname{(ess)}\sup_{x \in \R} \abs{u(x)}}\) in the (essentially) bounded~\({\textnormal{L}^\infty}\) case. Plancherel's theorem shows that~\({\mathscr{F}}\)~is an isometric isomorphism between \({\textnormal{L}^2}\) and~\({ \Set{ \widehat{  u  } \in \textnormal{L}^2(\R \to \C) \given \widehat{  u  }(- \xi) = \overline{ \widehat{  u  }( \xi)} } }\). Next define~\({\textnormal{H}^s}\), for any~\({s \geq 0}\), to be the fractional Sobolev space of functions in~\({\textnormal{L}^2}\) with finite norm~\({\norm{u}_{s} \coloneqq \norm*{\langle \cdot \rangle^s \widehat{u}}_{\textnormal{L}^2(\R \to \C)}}\) and inner product \({\innerproduct{u}{v}_s \coloneqq \int_{\R} \langle \cdot \rangle^{2s} \,\widehat{u} \, \overline{\widehat{v}} \dee \xi}\), where~\({\langle \xi \rangle = \sqrt{1 + \xi^2}}\), and write~\({\textnormal{L}^2}\) for~\({\textnormal{H}^0}\). Since~\({\langle \xi \rangle^s \eqsim 1 + \abs{\xi}^s}\), it follows, in the sense of weak \({\textnormal{L}^2}\)-derivatives, that~\({\| u \|_{s}^2 \eqsim  \| u \|_{0}^{2} + \| u^{(s)} \|_{0}^{2}}\) whenever~\({s \in \Z_+}\). In the fractional case~\({s = k + \sigma}\), with~\({k = \lfloor s \rfloor}\) and~\({\sigma \in (0, 1)}\), we also have the more \enquote{local}, finite-difference characterisation
\begin{equation*}
\label{eq:Hs-difference-norm}
\| u \|_{s}^2 \eqsim  \| u \|_{k}^2 + \int\limits_{\mathclap{|h| \leq \delta}} \bigl\| \upDelta_h^1 u^{(k)}\bigr\|_{0}^{2} \frac{ \dee h}{|h|^{1 + 2\sigma}}
\end{equation*}
where~\({\upDelta_h^1 f \coloneqq  f(\cdot + h) - f}\) and~\({\delta > 0}\) (commonly~\({\delta = \infty}\), but only behaviour around~\({h = 0}\) matters). All in all, we may therefore consider the space~\({\textnormal{H}^s(\Omega)}\) of real functions defined on an open set~\({\Omega \subset \R}\) whose norm equals that of~\({\textnormal{H}^s}\), except that \({\textnormal{L}^2}\)~integrals now go over~\({\Omega}\) (and with~\({\delta}\)~appropriately).

In the periodic case, given any~\({P > 0}\) and~\({q \geq 1}\), let~\({\textnormal{L}_P^q}\) be the space of \({P}\)-periodic, locally \({ q }\)-integrable functions with norm~\({\norm{u}_{\textnormal{L}_P^q} \coloneqq \bigl( \int_{- \frac{ P }{ 2 }}^{ \frac{ P }{ 2 }} \abs{u}^q \dee x \bigr)^{1/q}}\). In particular, \({u \in \textnormal{L}_P^2}\) has the Fourier-series representation~\({u = \sum_{\xi \in \Z}  \widehat{u}(\xi) \, e_\xi}\), now with~\({ \mathscr{F}}\) as an isomorphism~\({ \textnormal{L}_P^2 \to \Set{ \widehat{  u  } \in \ell^2(\Z) \given \widehat{  u  }( - \xi) = \overline{ \widehat{  u  }( \xi)} } }\), where
\begin{equation*}
 e_\xi(x) \coloneqq \frac{\textnormal{e}^{2 \uppi \textnormal{i} \xi x / P}}{\sqrt{P}} \qquad \text{and} \qquad \widehat{u}(\xi) = \innerproduct{u}{e_\xi}_{ \textnormal{L}_P^2} \coloneqq \int_{- \frac{ P }{ 2 }}^{ \frac{ P }{ 2 }} \!\!\! u \, \overline{e_\xi} \dee x.
\end{equation*}
Similarly as above, we introduce the \({P}\)-periodic real Sobolev space~\({\textnormal{H}_P^s}\), for~\({s \geq 0}\), with inner product \({ \innerproduct{u}{v}_{\textnormal{H}_P^s} \coloneqq \sum_{\xi \in \Z} \langle \xi \rangle_P^{2s} \, \widehat{u}(\xi) \, \overline{\widehat{v}(\xi)} }\) and norm~\({\norm{u}_{\textnormal{H}_P^s} \coloneqq \innerproduct{u}{u}_{\textnormal{H}_P^s}^{\frac{1}{2}}}\), where~\({\langle \xi \rangle_P \coloneqq \bigl\langle \tfrac{2 \uppi \xi}{P} \bigr\rangle}\). Again write~\({\textnormal{L}_P^2}\) for~\({\textnormal{H}_P^0}\) and note that
\begin{equation} \label{eq:periodicHs-difference-norm}
\| u \|_{ \textnormal{H}_P^s}^2 \eqsim  \| u \|_{\vphantom{ \textnormal{H}_P^s}\smash{\textnormal{H}_P^k}}^2 + \int\limits_{\mathclap{|h| \leq \delta}} \bigl\| \upDelta_h^1 u^{(k)}\bigr\|_{\vphantom{ \textnormal{H}_P^s}\smash{ \textnormal{L}_P^2}}^{2} \frac{\dee h}{|h|^{1 + 2\sigma}} 
\end{equation}
with~\({0 < \delta < \frac{P}{2}}\), omitting the last term if~\({s \in \Z_+}\). Thus~\({ \norm{ u }_{ \textnormal{H}_P^s } \eqsim \norm{ u }_{ \textnormal{H}^s \left( - \frac{P}{2}, \frac{P}{2} \right) } }\) for~\({ u \in \textnormal{H}^s_P }\). Moreover, for any~\({\varphi_P \in \textnormal{C}_{ \textnormal{c}}^{ \infty}(\R \to [0, 1]) }\) which is~\({ 1 }\) in~\({ \left( - \frac{P}{2}, \frac{P}{2} \right)  }\) and~\({ 0 }\) in~\({ \Set*{ \abs{ x }   \geq \frac{P}{2} + \tau } }\) for fixed~\({ \tau \lesssim P_{ \textnormal{min}} }\), we have
\begin{equation} \label{eq:localizing-HsP-in-Hs}
\norm{ u }_{ \textnormal{H}_P^s} \eqsim \norm{ \varphi_P u }_{ s }
\end{equation}
uniformly in~\({ P \geq P_{ \textnormal{min}}  > 0}\). \Cref{eq:localizing-HsP-in-Hs} demonstrates that~\({ \textnormal{H}_P^s }\) is locally in~\({ \textnormal{H}^s }\) and that results for~\({ \norm{  }_{ s }  }\) carry over to~\({ \norm{  }_{ \textnormal{H}_P^s }  }\)---in particular, we need not bother with the \({ P }\)-dependence in the hidden estimation constants. For example, when~\({s > \frac{1}{2}}\), there is a continuous embedding of~\({\textnormal{H}^s}\) into~\({ \textnormal{L}^\infty }\), and hence, \({\textnormal{H}_P^s \hookrightarrow \textnormal{L}^\infty}\)~also.

\subsection{Action of~\({L}\) on~\({ \textnormal{H}^s}\) and~\({ \textnormal{H}_P^s}\)} \label{sec:action-L}

It follows immediately from~\({\abs{\mathfrak{m}(\xi)} \lesssim \langle \xi \rangle^{\sigma}}\) that~\({L}\) maps~\({\textnormal{H}^s}\) continuously into~\({\textnormal{H}^{s + \abs{\sigma}}}\) for any~\({s}\). Its action on periodic spaces, however, is less trivial. If~\({\mathfrak{m} \in \textnormal{C}^\infty}\), then~\({ L}\) maps~\({\mathscr{S}}\) to itself and so it extends to a continuous operator~\({L \colon \mathscr{S}' \to \mathscr{S}'}\) still satisfying~\({\widehat{Lu} = \mathfrak{m} \, \widehat{u}}\). In~particular,
\begin{equation} \label{eq:action-m-periodic}
\widehat{Lu}(\xi) = \mathfrak{m}\bigl( \tfrac{2 \uppi \xi}{P} \bigr) \, \widehat{u}(\xi), \qquad \xi \in \Z,
\end{equation}
for \({P}\)-periodic distributions, so that~\({L \colon \textnormal{H}_P^s \to \textnormal{H}_P^{s + \abs{\sigma}}}\) continuously. Fortunately, there is a more direct approach to the periodic case which also works for irregular symbols in~\({ \textnormal{W}_0 }\).
\begin{proposition} \label{thm:convolution-periodic}
Convolution is a continuous bilinear operator~\({ \textnormal{L}^1 \ast \textnormal{L}_P^q \hookrightarrow \textnormal{L}_P^q }\) for all~\({q \in [1, \infty]}\). In~fact, if~\({f \in \textnormal{L}^1}\) and~\({u \in \textnormal{L}_P^q}\), then~\({f \ast u = f_P \ast_P u}\)~a.e., where
\begin{equation*}
f_P \coloneqq \sum_{j \in \Z} f(\cdot + jP) \in \textnormal{L}_P^1 \qquad \text{and} \qquad f_P \ast_P u \coloneqq  \int_{- \frac{P}{2}}^{ \frac{P}{2}} f_P(y) \, u(\cdot - y) \dee y.
\end{equation*}
Moreover,
\begin{equation} \label{eq:fourier-coeff-K_P}
\widehat{f_P}(\xi) = \sqrt{ \tfrac{2\uppi}{P}} \, \widehat{f}\bigl( \tfrac{2 \uppi \xi}{P} \bigr), \qquad \xi \in \Z,
\end{equation}
relating the Fourier coefficients of~\({ f_P }\) with the Fourier transform of~\({ f }\).
\end{proposition} \vspace*{-1em} 
\begin{proof}
Intuitively, we reduce~\({ \textnormal{L}^1 \ast \textnormal{L}_P^q \hookrightarrow \textnormal{L}_P^q }\) to a special case of~\({ \textnormal{L}_P^1 \ast_P^{} \textnormal{L}_P^q \hookrightarrow \textnormal{L}_P^q }\).  Since, in the most general case~\({ q = 1 }\),
\begin{equation*}
\int_{\R} \int_{- \frac{P}{2}}^{ \frac{P}{2}} \abs*{f(y) \, u(x - y) } \dee x \dee y = \norm{f}_{ \textnormal{L}^{1}} \norm{u}_{ \textnormal{L}_P^1} < \infty,
\end{equation*}
we find from the Fubini--Tonelli theorem that~\({f \ast u}\) exists a.e.\@ and is in~\({ \textnormal{L}_P^1}\). Subsequently we may then compute
\begin{align*} \SwapAboveDisplaySkip
f \ast u &= \sum_{j \in \Z} \int_{- \frac{P}{2}}^{ \frac{P}{2}} f(y + jP) \, u(\cdot - y) \dee y \\
&= \int_{- \frac{P}{2}}^{ \frac{P}{2}} \sum_{j \in \Z} f(y + jP) \, u(\cdot - y) \dee y  = f_P \ast_P u
\end{align*}
by dominated convergence, periodicity of~\({u}\) plus the fact that~\({f_P^{} \in \textnormal{L}_P^1}\). With this representation Young's inequality gives
\begin{equation*}
\norm{f \ast u}_{ \textnormal{L}_P^q} = \norm{f_P \ast_P u}_{ \textnormal{L}_P^q} \leq \norm{f_P}_{ \textnormal{L}_P^1} \norm{u}_{ \textnormal{L}_P^q},
\end{equation*}
and the result follows, noting that~\({ \norm{f_P}_{ \textnormal{L}_P^1} \leq \norm{f}_{ \textnormal{L}^{1}}}\). Similar reasoning also implies~\eqref{eq:fourier-coeff-K_P}.
\end{proof}
Directly from~\Cref{thm:convolution-periodic} and the convolution theorem for~\({\mathscr{F}}\) we then obtain the following result.

\begin{proposition}
\({L}\) is a Fourier multiplier on~\({\textnormal{L}_P^2}\) of the form~\eqref{eq:action-m-periodic}, mapping \({\textnormal{H}_P^s}\) to~\({\textnormal{H}_P^{s + \abs{\sigma}}}\) continuously.
\end{proposition}
Bear in mind that~\Cref{thm:convolution-periodic} is by no means true for general~\({f \in \textnormal{L}^1}\) if~\({\textnormal{L}_P^q}\) is replaced by~\({\textnormal{L}_{\textnormal{loc}}^q}\); it~is the periodic structure that saves~us.

\enlargethispage{2\baselineskip} 

\subsection{Cut-off argument and estimates for~\({n}\)} \label{sec:nonlinearity}

In studying~\eqref{eq:variational-lagrange-principle}, we will need that~\({n}\)---or more precisely, the induced operator~\({n(u)(x) \coloneqq n(u(x))}\)---is well-defined on~\({ \textnormal{H}^s \cap \textnormal{L}^\infty }\) and satisfies a~\enquote{fractional chain rule}. Specifically, the~following result~\autocite[Theorem~5.3.4/1~(i)]{RunstSickel1996} holds. Its~proof is based on a Taylor expansion of~\({n}\) and maximal-function techniques on dyadic scales to control the~remainder.

\begin{proposition}[Fractional chain rule] \label{thm:chain-rule-Hs}
Consider the case~\({ n \in \textnormal{Lip}_{ \textnormal{loc}} }\) or~\({ n \in  \textnormal{C}_\textnormal{loc}^1 }\) with~\({ s \in [0, 1] }\) in Assumption~\ref{assumption:nonlinearity}, or the case~\({ n \in \textnormal{C}_{ \textnormal{loc}}^{\varsigma} }\) with~\({ \varsigma \in (1, 1+ q) }\) and~\({s \in [0, \varsigma)}\). Then~\({n}\) induces a composition operator on~\({\textnormal{H}^{s} \cap B}\) satisfying
\begin{equation} \label{eq:chain-rule-Hs}
\norm{n(u)}_{s} \lesssim \norm{u}_{\infty}^{q} \norm{u}_{s},
\end{equation}
where~\({ B }\) is a sufficiently small ball around~\({ 0 }\) in~\({ \textnormal{L}^\infty }\). If~\({ n }\) is a monomial of order~\({ 1 + q \in \Z_+ }\), then~\eqref{eq:chain-rule-Hs} holds for all~\({ s \geq 0 }\).
\end{proposition}

\noindent Chain rule-type results with gaps between~\({s}\) and~\({1 + q}\) are common in the literature, \textit{e.g.}~\autocite[Section~3]{ChrWei1991a}, but it does not seem to be well known that one can let~\({s}\) be arbitrarily close to the regularity index of the outer function. 

Since we shall find solitary waves from the periodic problem as~\({ P \to \infty }\), it is very important that~\eqref{eq:chain-rule-Hs} extends to~\({ \textnormal{H}_P^s }\) and holds \emph{uniformly} in~\({ P \geq P_{ \textnormal{min}} }\). Estimating
\begin{equation} \label{eq:chain-rule-HsP-proof}
\norm{n(u)}_{\textnormal{H}_P^s} \eqsim \norm{n(\varphi_P u)}_{s} \lesssim  \norm{ \varphi_P u}_{\infty}^{q} \norm{\varphi_P u}_{s} \eqsim \norm{u}_{\infty}^{q} \norm{u}_{\textnormal{H}_P^s}
\end{equation}
with help of~\eqref{eq:localizing-HsP-in-Hs}, shows that this is indeed the~case. The~first equivalence is a natural extension of~\eqref{eq:localizing-HsP-in-Hs} and proved in the same fashion using Leibniz' rule (\({ \lfloor s \rfloor }\)~times) plus the fact that~\({ \norm{ \varphi_P^{(k)} }_{ \infty } \lesssim \tau^{-k} \lesssim 1 }\) uniformly in~\({ P }\).

\begin{corollary} [Fractional chain rule on~\({ \textnormal{H}_P^s }\)] \label{thm:chain-rule-HsP}
Suppose under Assumption~\ref{assumption:nonlinearity} that \({ n \in \textnormal{Lip}_{ \textnormal{loc}} }\) or~\({ n \in  \textnormal{C}_\textnormal{loc}^1 }\) with~\({ s \in [0, 1] }\), or~\({ n \in \textnormal{C}_{ \textnormal{loc}}^{\varsigma} }\) with~\({ \varsigma \in (1, 1+ q) }\) and~\({s \in [0, \varsigma)}\). Then~\({ n }\) induces a composition operator on~\({ \textnormal{H}_P^s \cap B }\) satisfying, uniformly in~\({ P }\) bounded away from~\({ 0 }\),
\begin{equation} \label{eq:chain-rule-HsP}
\norm{ n(u) }_{ \textnormal{H}_P^s } \lesssim \norm{ u }_{ \infty }^q \norm{ u }_{ \textnormal{H}_P^s },
\end{equation}
where~\({ B }\) is a sufficiently small ball around~\({ 0 }\) in~\({ \textnormal{L}^\infty }\). If~\({ n }\) is a monomial of order~\({ 1 + q \in \Z_+ }\), then~\eqref{eq:chain-rule-HsP} holds for all~\({ s \geq 0 }\).
\end{corollary}

In the \textit{a~priori} unbounded case~\({s \leq \frac{1}{2}}\), we~also cut off the growth of~\({n}\) and consider instead
\begin{equation} \label{eq:n-cutoff}
\widetilde{n}(x) =
\begin{dcases}
n(x) & \text{if } \abs{x} \leq A_{\mu}; \\
n(A_{\mu} \sign x) & \text{if } \abs{x} > A_{\mu},
\end{dcases}
\end{equation}
where~\({A_{\mu} \sim \mu^\theta}\) and~\({\theta \in \left(0, \tfrac{1}{2} \right)}\). Then
\begin{equation} \label{eq:n-cutoff-estimate}
\abs{\widetilde{n}(x)} \lesssim \mu^{\theta q} \abs{x}
\end{equation}
for \emph{all}~\({x \in \R}\) for~\({\mu}\) sufficiently small. Moreover, now~\({ \widetilde{n}}\) is globally Lipschitz and satisfies, directly from~\eqref{eq:periodicHs-difference-norm},
\begin{equation*}
\norm{\widetilde{n}(u)}_{ \textnormal{H}_P^s} \lesssim \mu^{\theta q} \norm{u}_{ \textnormal{H}_P^s}.
\end{equation*}
This estimate mimics the fractional chain rule~\eqref{eq:chain-rule-HsP} up to a small loss in the exponent~\({ q }\). We shall obtain that~\({ \norm{u_P^\star}_{\infty}^2 \lesssim \mu }\) for solutions~\({u_P^\star}\) of the modified variational problem with~\({\widetilde{n}}\) replaced by~\({n}\). Therefore, since~\({\theta < \frac{1}{2}}\), we get~\({\norm{u_P^\star}_{\infty} \leq A_{\mu}}\) for all sufficiently small~\({\mu}\). In other words,~\({\widetilde{n}(u_P^\star) = n(u_P^\star)}\), and so~\({u_P^\star}\) in fact solves the original problem. For~the sake of brevity, write~\({n}\) for~\({\widetilde{n}}\) from now~on.

\Cref{thm:chain-rule-Hs} and \Cref{thm:chain-rule-HsP} naturally restrict the range of feasible~\({ s }\) from above. As~regards a~lower bound, we need~\({ u_P^\star \in \textnormal{L}^{ \infty}}\). By~construction~\({ n(u_P^\star) \in \textnormal{L}^\infty }\), and so from~\eqref{eq:traveling-wave} it suffices that~\({ L u_P^\star \in \textnormal{L}^{ \infty } }\). This~follows whenever~\({ s  > \frac{1}{2} - \abs{ \sigma } }\) in light of~\({ L \colon \textnormal{H}_P^s \to \textnormal{H}_P^{\vphantom{ s}\smash{s + \abs{ \sigma }}} }\). Furthermore,~\eqref{eq:traveling-wave} also yields
\begin{equation} \label{eq:Linfty-HsP-trick}
\bigl( \nu_P^{} - \mu^{ \theta q} \bigr) \norm{ u_P^\star }_{ \infty } \lesssim  \norm{ Lu_P^\star }_{ \infty } \lesssim \norm{ Lu_P^\star }_{\vphantom{ \textnormal{H}_P^s}\smash{ \textnormal{H}_P^{\vphantom{s}\smash{s + \abs{ \sigma }}}} } \lesssim  \norm{ u_P^\star }_{ \textnormal{H}_P^s }.
\end{equation}
Hence, as we will establish that~\({\!\nu_P }\) is uniformly bounded away from~\({ 0 }\) and \({ \norm{ u_P^\star }_{ \textnormal{H}_P^s }^2 \lesssim \mu }\) in~\Cref{thm:lower-bound-wave-speed-periodic,thm:bound-HsP-norm-mu}, this gives~\({ \norm{u_P^\star}_{\infty}^2 \lesssim \mu }\) for all sufficiently small~\({ \mu }\). Similar reasoning applies in the solitary-wave~case.

\subsection{Properties of functionals} \label{sec:functionals}

Finally, we list some basic features of~\({ \mathcal{L} }\), \({ \mathcal{N} }\), \({ \mathcal{Q} }\) and their periodic counterparts. By weak continuity of an operator we mean that the operator maps weakly convergent sequences to strongly convergent sequences, which in the result below follows from the compact embedding of~\({ \textnormal{H}_P^s }\) in~\({ \textnormal{H}_P^t }\) whenever~\({s > t }\).

\begin{proposition} \label{thm:functional-props}
If~\({ s \geq 0 }\), then~\({ \mathcal{L}, \mathcal{Q}, \mathcal{N} \in \textnormal{C}^{ 1 }( \textnormal{H}^s \to \R ) }\) and~\({ \mathcal{L}_P^{}, \mathcal{N}_P^{}, \mathcal{Q}_P^{} \in \textnormal{C}^1( \textnormal{H}_P^s \to \R) }\) have \({ \textnormal{L}^2 }\) and~\({ \textnormal{L}_P^2 }\) derivatives, respectively, given~by
\begin{equation*}
\mathcal{L}'(u) \coloneqq - L u, \qquad \mathcal{N}'(u) \coloneqq - n(u) \qquad \text{and} \qquad \mathcal{Q}'(u) \coloneqq u.
\end{equation*}
Moreover, if~\({ s > 0 }\), then \({\calL_P}\),~\({\calN_P}\) and thus also~\({\calE_P}\) are weakly continuous on~\({\textnormal{H}_P^s}\).
\end{proposition}

\section{Penalised variational problem for periodic traveling waves} \label{sec:periodic}

\begin{wrapfigure}{r}{.49\textwidth}%
\vspace{-10pt}%
\begin{center}
\includegraphics{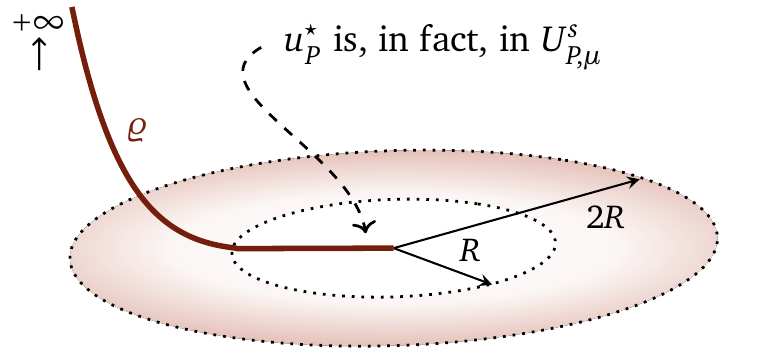}
\caption{Illustrating the penalised problem.}
\label{fig:penalization}
\end{center}\vspace{-.5em}%
\end{wrapfigure}
In this section we prove~\Cref{thm:existence-periodic} by finding a constrained local minimiser of~\({\calE_{P}}\) satisfying the Lagrange multiplier principle. Specifically, we look for a minimiser \({u_P^\star}\) in the set
\begin{equation*}
U_{P,\mu}^{s} \coloneqq U_{P, \mu}^{s, R} \coloneqq  \Set*{ u \in \textnormal{H}_P^s \given \norm{u}_{\textnormal{H}_P^s} < R \text{ and } \calQ_P(u) = \mu}
\end{equation*}
for which \({ \calE_P'(u_P^\star) + \nu_P^{} \calQ_P'(u_P^\star) = 0 }\) for a multiplier~\({\!\nu_P \in \R}\). Since~\({ \mathcal{E}_P }\) is noncoercive, however, minimising sequences may approach the \enquote{boundary} \({ \norm{  }_{ \textnormal{H}_P^s } = R }\) of~\({ U_{P, \mu}^s }\), where Lagrange's principle might fail. In~order to resolve this issue, we introduce a smooth, increasing \emph{penaliser}~\({\varrho \colon \left[0, (2R)^2  \right) \to [0, \infty)}\) satisfying
\begin{equation*}
\varrho(t) = 0  \text{ when }  0 \leq t \leq R^2 \qquad \text{and} \qquad \varrho(t) \nearrow \infty \text{ as } t \nearrow (2R)^2,
\end{equation*}
and instead minimise
\begin{equation*}
\calE_{P, \varrho}(u) \coloneqq \calE_P(u) + \varrho \left( \norm{u}_{\textnormal{H}_P^s}^2 \right)
\end{equation*}
over the larger set~\({\widetilde{U}_{P, \mu}^{} \coloneqq U_{P, \mu}^{s, 2R}}\); see \cref{fig:penalization}. For technical reasons, we also assume that for every \({a \in (0, 1)}\) there exists~\({b > 1}\) such that
\begin{equation} \label{eq:penalizer-derivative}
\varrho'(t) \lesssim \varrho(t)^a + \varrho(t)^b
\end{equation}
for all~\({t \in \left[ R^2, (2R)^2  \right)}\). An example~\autocite[Section~3]{EGW2012}, up to appropriate scaling, is given by
\begin{equation*}
t \mapsto
\begin{dcases}
\frac{\mathrm{e}^{-1/(t - R^2)}}{(2R)^2 - t} & \text{if } t \in \big(R^2, (2R)^2\big); \\
0 & \text{if } t \in \big[0, R^2 \big].
\end{dcases}
\end{equation*}
\textit{A~priori} estimates below show that~\({\varrho}\) is inactive at the minimum, and hence~\({u_P^\star \in U_{P, \mu}^{s}}\), as desired. 

\begin{lemma} \label{thm:penalized-existence}
\({\calE_{P, \varrho}}\) admits a minimiser~\({u_P^\star \in \widetilde{U}_{P, \mu}^{}}\) satisfying the Euler--Lagrange equation
\begin{equation} \label{eq:penalized-eulerlagrange}
\innerproduct*{Lu_P^\star + n(u_P^\star) - \nu_P^{} u_P^\star}{w}_{\textnormal{L}_P^2} = 2\varrho' \bigl( \norm{u_P^\star}_{\textnormal{H}_P^s}^2 \bigr) \innerproduct*{u_P^\star}{w}_{\textnormal{H}_P^s}
\end{equation}
for all~\({w \in \textnormal{H}_P^s}\), where~\({\!\nu_P \in \R}\) is the multiplier. If~\({\varrho' > 0}\), then~\({u_P^\star \in \textnormal{H}_P^{3s}}\).
\end{lemma} \vspace*{-1.5em} 
\begin{proof}
Since~\({\varrho}\) is weakly lower semi-continuous and coercive, so is~\({\calE_{P, \varrho}}\) by \Cref{thm:functional-props}. Hence, it suffices to search for minimisers in the subset~\({\Set{u \in \widetilde{U}_{P, \mu} \given \norm{u}_{\textnormal{H}_P^s} \leq R'}}\) for some~\({R' < 2R}\). This set is weakly closed by the compact embedding~\({\textnormal{H}_P^s \hookrightarrow \textnormal{L}_P^2}\) for~\({s > 0}\) together with the fact that closed balls are weakly closed (a consequence of Mazur's~lemma). Existence of a minimiser~\({u_P^\star}\) now follows from the generalised extreme value theorem~(\autocite[theorem~1.2]{Struwe2008}). Evaluating
\begin{equation*}
\dualitypairing*{\calU_P'(u_P^\star)}{u_P^\star}_{ \textnormal{L}_P^2} = 2\, \calU_P(u_P^\star) > 0
\end{equation*}
shows that \({\dualitypairing{ \calU_P'(u_P^\star)}{ \cdot}_{ \textnormal{L}_P^2} }\) does not vanish identically, and so Lagrange's principle gives~\eqref{eq:penalized-eulerlagrange}.

As regards regularity, note that~\eqref{eq:penalized-eulerlagrange} especially holds for all~\({w}\) in the Fourier basis, implying that 
\begin{equation} \label{eq:penalized-eulerlagrange-discrete}
\widehat{Lu_P^\star} + \widehat{n(u_P^\star)} - \nu_P^{} \widehat{u_P^\star} = 2\varrho' \bigl( \norm{u_P^\star}_{\textnormal{H}_P^s}^2 \bigr) \langle \cdot \rangle_P^{2s} \widehat{u_P^\star}
\end{equation}
pointwise in~\({\Z}\). Since~\({u_P^\star, Lu_P^\star, n(u_P^\star) \in \textnormal{H}_P^s}\), we get~\({\langle \cdot \rangle_P^{2s} \widehat{u_P^\star} \in \scrF (\textnormal{H}_P^s)}\) if~\({\varrho' > 0}\), that is,~\({u_P^\star \in \textnormal{H}_P^{3s}}\).
\end{proof}

Perhaps~\({u_P^\star}\) is just a constant solution of~\eqref{eq:traveling-wave}? Due to the constraint~\({ \mathcal{Q}_P(u) = \mu }\), such solutions, if they exist, can only be of the form~\({u_\textnormal{trivial} \coloneqq  \pm \sqrt{2\mu / P}}\). Inserting \({ u_{ \textnormal{trivial}} }\) into~\eqref{eq:traveling-wave} gives
\begin{equation*}
(\nu_P - \mathfrak{m}(0)) \; u_{ \textnormal{trivial}} = n(u_{ \textnormal{trivial}}),
\end{equation*}
and since~\({ n }\) is superlinear near the origin, we observe that \({ u_{ \textnormal{trivial}} }\) will solve~\eqref{eq:traveling-wave} when~\({ \!\nu_P > \mathfrak{m}(0) }\) for suitable~\({ \mu }\) and~\({ P }\) with~\({ u_{ \textnormal{trivial}} }\) small enough. In~fact, constant solutions may also exist at subcritical speeds~\({ \!\nu_P < \mathfrak{m}(0) }\)---for example if~\({ u_{ \textnormal{trivial}} < 0 }\) and~\({ n(x) \equiv n_q(x) = \gamma \abs{ x }^{1+q} }\), with~\({ \gamma > 0 }\). Fortunately, however, \Cref{thm:functional-penalized-inf-estimate} demonstrates that \({ u_{ \textnormal{trivial}} }\) does not minimise~\({ \mathcal{E}_{ P, \varrho} }\) for sufficiently small~\({\mu}\) and large~\({P}\).

\begin{lemma} \label{thm:jensen-estimate}
For all \({q > 0}\) it is true that
\begin{equation*}
\Gamma_q \coloneqq \frac{1}{2\uppi} \int_{-\uppi}^{\uppi} \left( \sqrt{\tfrac{2}{3}} \left( 1 + \sin x \right) \right)^{2 + q}  \dee x > 1.
\end{equation*}
\end{lemma} \vspace*{-1em} 
\begin{proof}
Define \({ f(x) = \left( \sqrt{\tfrac{2}{3}} \left( 1 + \sin x \right) \right)^2 }\) and~\({ \varphi(x) = x^{(2 + q)/2} }\). Then Jensen's inequality with strict convexity gives
\begin{equation*}
\Gamma_q = \frac{1}{2\uppi} \int_{-\uppi}^{\uppi} \!\!\! \varphi \big( f(x) \big) \dee x > \varphi \left( \frac{1}{2\uppi} \int_{-\uppi}^{\uppi} \!\!\! f(x) \dee x \right) = \varphi(1) = 1.
\end{equation*}
\end{proof}

\enlargethispage{2\baselineskip} 

\begin{lemma}
\label{thm:functional-penalized-inf-estimate}
For all sufficiently small~\({ \mu > 0 }\) there exists~\({P_\mu > 0}\) such that \({ u_{ \textnormal{trivial}} }\) does not minimise~\({\mathcal{E}_{P,\varrho}}\) on~\({\widetilde{U}_{\mu,P}}\) and
\begin{equation}
\label{eq:Eprho-inf-estimate}
\inf \Set*{\mathcal{E}_{P,\varrho}(u) \given u \in \widetilde{U}_{P, \mu}} <- \mu \left[ \mathfrak{m}(0) + C \left( \frac{2\mu}{P} \right)^{q/2} \right]
\end{equation}
whenever~\({ P \geq P_\mu  }\), where~\({ C > 0}\). If~\({n = n_q}\), we explicitly have~\({C = 2 \abs{\gamma} / (2 + q)}\).
\end{lemma}
\begin{proof}
Constructively,
\begin{equation*}
u(x) \coloneqq A \sign(\gamma) \sqrt{\tfrac{2}{3}}\Bigl[ 1 + \sin \left(\tfrac{2 \uppi x}{P}\right) \Bigr],
\end{equation*}
scaled to obey~\({\mathcal{U}_P(u) = \mu}\), where~\({A \coloneqq \sqrt{2 \mu / P}}\), will be shown to satisfy both
\begin{equation}
\label{eq:Esinus_estimate}
\mathcal{E}_P(u) < - \mu \left( \mathfrak{m}(0) + C \, A^q \right) \qquad \text{and} \qquad \mathcal{E}_P(u) < \mathcal{E}_P(u_\textnormal{trivial})
\end{equation}
for suitable~\({\mu}\),~\({P}\), and~\({C > 0}\). As~\({u}\) lies in~\({U_{P, \mu}^s}\), where~\({\mathcal{E}_{P,\varrho} \equiv \mathcal{E}_P}\), for sufficiently small~\({\mu}\), this proves the~claim. (Note that it suffices to only consider positive~\({u_\textnormal{trivial}}\), because~\({\calE_{P, \varrho}(A) \leq \calE_{P, \varrho}(-A)}\).)

Indeed,
\begin{equation*}
 \mathcal{E}_P(A \sign\gamma ) = - \mu \left[\mathfrak{m}(0) + \tfrac{2 |\gamma|}{2 + q} A^q  + \smalloh (A^q)\right],
\end{equation*}
and 
\begin{equation*}
\mathcal{E}_P(u_\textnormal{trivial}) \geq \mathcal{E}_P(A \sign \gamma)
\end{equation*}
provided~\({A}\) is sufficiently small (this condition safeguards a possible issue when~\({n_q(x) = \gamma |x|^q}\) and~the signs of~\({u_\textnormal{trivial}}\) and~\({\gamma}\) coincide). Nonzero Fourier coefficients of~\({u}\) are~\({\widehat{u}_0 = 2\sqrt{\mu / 3}}\) and~\({|\widehat{u}_{\pm 1}| =  \sqrt{\mu / 3}}\), so that~\({\| u \|_{s}^2 = \tfrac{2}{3}\mu \left( 2 + \langle 1 \rangle_P^{2s} \right)}\) is controlled by~\({\mu}\). Moreover, expanding~\({ \mathfrak{m} }\) gives that
\begin{align*}
 \mathcal{L}_P(u) &= - \mu \left[\tfrac{2}{3} \mathfrak{m}(0) + \tfrac{1}{3}\mathfrak{m}\left( \tfrac{2 \uppi}{P} \right) \right] \\
 &= - \mu \left[ \mathfrak{m}(0) + c P^{-2\ell} + \mathcal{O}\left( P^{-2\ell - 2} \right)  \right]
\end{align*}
for~\({c \coloneqq \mathfrak{m}^{(2\ell)}(0) / (2\ell)! < 0}\). With~\({\Gamma_q}\) from~\Cref{thm:jensen-estimate}, this yields, after a change of variables in~\({\calN_P(u)}\), that
\begin{equation*}
\mathcal{E}_P(u) = - \mu \left[ \mathfrak{m}(0) + c P^{-2\ell} + \mathcal{O}\left( P^{-2\ell - 2} \right)  + \tfrac{2|\gamma|}{2 + q}\Gamma_q  \,A^q + \smalloh (A^q) \right].
\end{equation*}
Consequently, the first inequality in~\eqref{eq:Esinus_estimate} then holds for~\({A }\) sufficiently small, while, since~\({\Gamma_q > 1}\) and~\({q < 4 \ell}\), the second inequality becomes true for~\({A}\) sufficiently small and~\({P}\) large enough.
\end{proof}
\vspace*{-.8em} 

\begin{remark}
Bound~\eqref{eq:Eprho-inf-estimate} has not optimal order with respect to~\({ q }\) and has the defect of depending on~\({ P }\). By~comparing with the solitary-wave problem, however, we can do better; see~\Cref{thm:solitary-inf-upper-bound}.
\end{remark}

Closely based on~\autocite[Lemmas~3.5–6]{EGW2012} we next establish that~\({\varrho' \bigl( \norm{u_P^\star}_{\textnormal{H}_P^s}^2 \bigr)}\) eventually vanishes based on a lower bound on~\({\!\nu_P}\) and an~\textit{a~priori} estimate for~\({\norm{u_P^\star}_{\textnormal{H}_P^s}}\).

\begin{lemma} \label{thm:lower-bound-wave-speed-periodic}
With~\({ \mu }\) and~\({ P_\mu}\) as in~\Cref{thm:functional-penalized-inf-estimate}, the estimate
\begin{equation} \label{eq:lower-bound-wave-speed-periodic}
\nu_P - \mathfrak{m}(0) > \widetilde{C} \left( \frac{2\mu}{P} \right)^{q/2} - c_{\varrho} \mu^{\lambda} + \left\{
\begin{aligned}
& \mathcal{O}\bigl(\mu^{ \theta q}\bigr) && \text{if } s \leq \tfrac{ 1 }{ 2 } \\
& o\bigl(\norm{ u_P^\star }_{ \infty }^{q}\bigr)  && \text{if } s > \tfrac{ 1 }{ 2 }
\end{aligned}
\right\}
\end{equation}
holds over the set of minimisers~\({u_P^\star}\) of~\({\calE_{P, \varrho}^{}}\) over~\({\widetilde{U}_{P, \mu}^{}}\) and~\({P \geq P_\mu}\). Here~\({\widetilde{C} > 0 }\) (equals~\({ \abs{ \gamma } }\) if~\({ n = n_q }\)), \({  \lambda > 0}\), and~\({ c_{\varrho} \geq 0}\) vanishes when~\({\varrho = 0}\).
\end{lemma}
\vspace*{-.8em} 
\begin{proof}
Write~\({u \coloneqq u_P^\star}\) for clarity. We shall obtain~\eqref{eq:lower-bound-wave-speed-periodic} using the identity
\begin{equation} \label{eq:wave-speed-periodic-proof-identity}
\innerproduct*{Lu + n(u)}{u}_{\textnormal{L}_P^2} = -(2 + q) \calE_P(u) + q \mathcal{L}_P(u) - \int_{-\frac{P}{2}}^{\frac{P}{2}} \bigl[ (2 + q) N(u) - u n(u) \bigr] \dee x,
\end{equation}
where the last integral vanishes if~\({ n }\) is homogeneous.

First choose~\({w = u}\) in~\eqref{eq:penalized-eulerlagrange} and observe that
\begin{equation*}
2 \nu_P \mu \geq \innerproduct*{Lu + n(u)}{u}_{\textnormal{L}_P^2} - \varrho' \bigl( \norm{u}_{\textnormal{H}_P^s}^2 \bigr) \cdot 4 R^2.
\end{equation*}
Since
\begin{equation*}
-\calE_P(u) = - \calE_{P, \varrho}(u)  + \varrho \bigl( \norm{u}_{\textnormal{H}_P^s}^{2} \bigr) >  \mu \left[ \mathfrak{m}(0) + C \left( \frac{2\mu}{P} \right)^{q/2} \right]
\end{equation*}
by~\eqref{eq:Eprho-inf-estimate} and~\({\varrho \geq 0}\), and~\({ \mathcal{L}_P(u) \geq - \mathfrak{m}(0) \mu }\), we deduce from~\eqref{eq:wave-speed-periodic-proof-identity} that
\begin{equation*}
\nu_P - \mathfrak{m}(0) > \underbrace{\tfrac{ 2 + q }{ 2 } C}_{ \eqqcolon \widetilde{ C }} \left( \frac{2\mu}{P} \right)^{q/2} - \mu^{-1} \varrho' \bigl( \norm{u}_{\textnormal{H}_P^s}^2 \bigr) \cdot 4 R^2 + \left\{
\begin{aligned}
& \mathcal{O}\bigl(\mu^{ \theta q}\bigr) && \text{if } s \leq \tfrac{ 1 }{ 2 } \\
& o\bigl(\norm{ u}_{ \infty }^{q}\bigr)  && \text{if } s > \tfrac{ 1 }{ 2 }
\end{aligned}
\right\},
\end{equation*}
because
\begin{equation*}
(2 + q) N(u(x)) - u(x) n(u(x)) = \left\{
\begin{aligned}
&  \mathcal{O}\bigl(\abs{ u(x) }^2 \mu^{ \theta q}\bigr) && \text{if } s \leq \tfrac{ 1 }{ 2 } \\
& o\bigl(\abs{ u(x) }^{ 2 + q }\bigr)  && \text{if } s > \tfrac{ 1 }{ 2 }
\end{aligned}
\right\}
\end{equation*}
uniformly over~\({ u \in \widetilde{U}_{P, \mu} }\) and~\({ x \in \R }\), where we used~\eqref{eq:n-cutoff-estimate} when~\({ s \leq \frac{1}{2} }\).

It remains to establish that~\({ \varrho' \bigl( \norm{u}_{\textnormal{H}_P^s}^2 \bigr) \lesssim \mu^{1 + \lambda} }\) for some~\({ \lambda > 0 }\), and using~\eqref{eq:penalizer-derivative}, it suffices to prove that~\({ \varrho \bigl( \norm{u}_{\textnormal{H}_P^s}^2 \bigr) \lesssim \mu^{1 + \tilde{ \lambda }} }\) for some~\({ \tilde{ \lambda } > 0 }\). Crudely, we have~\({ \mathcal{E}_{P, \varrho} (u) < - \mu \mathfrak{m}(0)}\), and so
\begin{equation*}
\varrho \bigl( \norm{u}_{\textnormal{H}_P^s}^2 \bigr) < - \mu \mathfrak{m}(0) - \mathcal{L}_P(u) - \mathcal{N}_P(u) \leq - \mathcal{N}_P(u).
\end{equation*}
If~\({ s \leq \frac{1}{2} }\), then~\({ - \mathcal{N}_P(u) \lesssim \mu^{1 + \theta q} }\) directly from~\({ \abs{ N(x) } \lesssim \mu^{ \theta q} \abs{ x }^2 }\). In case~\({ s > \frac{1}{2}}\), then~\({ - \mathcal{N}_P(u) \lesssim \mu\norm{ u }_{ \infty }^{ q } }\). Choose~\({ \vartheta \in (0, 1) }\) such that~\({ \widetilde{ s } \coloneqq (1 - \vartheta)s \in \bigl(\frac{1}{2}, s \bigr) }\). By interpolation,
\begin{equation} \label{eq:Linfty-estimate-interpolation-L2}
\norm{ u }_{ \infty }^{} \lesssim \norm{ u }_{ \vphantom{ \textnormal{H}_P^s}\smash{\textnormal{H}_P^{\widetilde{ s }}} }^{} \leq \norm{ u }_{ \vphantom{ \textnormal{H}_P^s}\smash{\textnormal{L}_P^2} }^{ \vartheta } \norm{ u }_{ \textnormal{H}_P^s }^{ 1- \vartheta} \lesssim \norm{ u }_{ \vphantom{ \textnormal{H}_P^s}\smash{\textnormal{L}_P^2} }^{ \vartheta}
\end{equation}
uniformly over~\({ u \in \widetilde{U}_{P, \mu} }\) and~\({ P \geq P_\mu }\), from which it follows that~\({ \varrho \bigl( \norm{u}_{\textnormal{H}_P^s}^2 \bigr) \lesssim \mu^{1 + \vartheta q} }\).
\end{proof}

\begin{lemma} \label{thm:bound-HsP-norm-mu}
The estimate
\begin{equation*}
\norm{u_P^\star}_{\textnormal{H}_P^s} \eqsim \mu^{ \frac{ 1 }{ 2 }}
\end{equation*}
holds uniformly over the set of minimisers of~\({ \mathcal{E}_{P, \varrho} }\) over~\({ \widetilde{ U }_{P, \mu} }\) and~\({ P \geq P_\mu }\).
\end{lemma}
\begin{proof}
Let~\({u \coloneqq u_P^\star}\) for convenience. Using~\({w \coloneqq \scrF^{-1} \bigl( \langle \cdot \rangle_P^{2s} \widebar{\widehat{u}} \bigr) \in \textnormal{H}_P^s}\) in~\eqref{eq:penalized-eulerlagrange} if~\({\varrho' > 0}\), or multiplying~\eqref{eq:penalized-eulerlagrange-discrete} by~\({\langle \cdot \rangle_P^{2s} \widebar{\widehat{u}}}\) and summing over~\({\Z}\) if~\({\varrho' = 0}\), we find---with the strong zero-convention (\({0 \cdot \infty = 0}\))---that
\begin{align*}
\nu_P \norm{u}_{\textnormal{H}_P^s}^{2} &= \innerproduct*{Lu + n(u)}{u}_{\textnormal{H}_P^s} - 2\varrho' \bigl( \norm{u}_{\textnormal{H}_P^s}^{2} \bigr) \norm{u}_{\vphantom{ \textnormal{H}_P^s}\smash{\textnormal{H}_P^{2s}}}^{2} \\
&\leq \norm{u}_{\vphantom{ \textnormal{H}_P^s}\smash{\textnormal{H}_P^{\smash[b]{s + \frac{\sigma}{2}}}}}^{2} + \norm*{n(u)}_{\textnormal{H}_P^s} \norm{u}_{\textnormal{H}_P^s},
\end{align*}
because \({ \abs{\innerproduct{Lu}{u}_{\textnormal{H}_P^s}} \lesssim \norm{u}_{\vphantom{ \textnormal{H}_P^s}\smash{\textnormal{H}_P^{\smash[b]{s + \frac{\sigma}{2}}}}}^{2}
 }\) by assumption on~\({ \mathfrak{m}}\). If~\({ s > \frac{ 1 }{ 2 } }\), the fractional chain rule (\Cref{thm:chain-rule-HsP}) and~\eqref{eq:Linfty-estimate-interpolation-L2} imply
\begin{equation*}
\norm*{n(u)}_{\textnormal{H}_P^s} \lesssim \norm{u}_{\infty}^q \norm{u}_{\textnormal{H}_P^s} \lesssim \mu^{ \vartheta q / 2} \norm{u}_{\textnormal{H}_P^s},
\end{equation*}
while if~\({ s \leq \frac{ 1 }{ 2 } }\), then
\begin{equation*}
\norm*{n(u)}_{\textnormal{H}_P^s} \lesssim \mu^{ \theta q} \norm{ u }_{ \textnormal{H}_P^s }.
\end{equation*}
From~\Cref{thm:lower-bound-wave-speed-periodic}, combined with~\eqref{eq:Linfty-estimate-interpolation-L2}~when~\({ s > \frac{ 1 }{ 2 } }\), we find that~\({\!\nu_P }\) is uniformly bounded away from~\({ 0 }\) for all sufficiently small~\({ \mu }\), uniformly over the set of minimisers of~\({ \mathcal{E}_{P, \varrho} }\) over~\({ \widetilde{ U }_{P, \mu} }\) and~\({ P \geq P_\mu }\). Hence, with~\({ \mu }\) possibly even smaller,
\begin{equation*}
\norm{ u }_{ \textnormal{H}_P^s }^{ 2 } \lesssim \norm{u}_{\vphantom{ \textnormal{H}_P^s}\smash{\textnormal{H}_P^{\smash[b]{s + \frac{\sigma}{2}}}}}^{2}.
\end{equation*}
Interpolating
\begin{equation*}
\norm{u}_{\vphantom{ \textnormal{H}_P^s}\smash{\textnormal{H}_P^{\smash[b]{s + \frac{\sigma}{2}}}}}^{2} \leq \norm{ u }_{\vphantom{ \textnormal{H}_P^s}\smash{ \textnormal{L}_P^2 }}^{ \abs{ \sigma } / s} \norm{ u }_{ \textnormal{H}_P^s }^{ 2 - (\abs{ \sigma } / s) }
\end{equation*}
if~\({ \sigma > -2s }\), or using that \({ \norm{u}_{\vphantom{ \textnormal{H}_P^s}\smash{\textnormal{H}_P^{\smash[b]{s + \frac{\sigma}{2}}}}}^{2} \leq \norm{ u }_{\vphantom{ \textnormal{H}_P^s}\smash{\textnormal{L}_P^2}}^2 }\) if~\({ \sigma \leq -2s }\), then gives~\({  \norm{ u }_{ \textnormal{H}_P^s } \lesssim \norm{ u }_{ \textnormal{L}_P^2 } }\), and in combination with~\({ \norm{ u }_{ \textnormal{H}_P^s } \geq \norm{ u }_{ \textnormal{L}_P^2 } }\) and~\({ \norm{ u }_{ \textnormal{L}_P^2 } =  (2\mu)^{ \frac{ 1 }{ 2 }} }\), this concludes the proof. 
\end{proof}

According to~\Cref{thm:bound-HsP-norm-mu}, \({ \varrho }\) vanishes for sufficiently small~\({ \mu }\), and so~\({ u_P^\star }\) is in fact a minimiser for~\({ \mathcal{E}_P }\) over~\({ U_{P, \mu}^s }\) satisfying \({ \norm{ u_P^\star }_{ \infty } \lesssim \norm{ u_P^\star }_{ \textnormal{H}_P^s } \eqsim \mu^{ \frac{1}{2}} }\), where we remember estimate~\eqref{eq:Linfty-HsP-trick}. In~particular, \({ u_P^\star }\) solves~\eqref{eq:traveling-wave} with wave speed~\({ \!\nu_P }\), noting that
\begin{equation}  \label{eq:wavespeed-est-from-equation}
\nu_P^{} - \mathfrak{m}(0) \lesssim \norm{ u_P^\star }_{ \infty }^{ q } \lesssim \mu^{q/2}
\end{equation}
uniformly over~\({ P \geq P_\mu }\), which follows from
\begin{equation*}
( \nu_P^{} - \mathfrak{m}(0) ) \norm{ u_P^\star }_{ \textnormal{L}_P^2 } \leq \norm{ ( \nu_P^{} - L ) u_P^\star }_{ \textnormal{L}_P^2 } = \norm{ n(u_P^\star) }_{ \textnormal{L}_P^2 } \lesssim \norm{ u_P^\star }_{ \infty }^{ q } \norm{ u_P^\star }_{ \textnormal{L}_P^2 } .
\end{equation*}

In order to finish~\Cref{thm:existence-periodic}, it remains to establish the improved bounds on~\({ \!\nu_P }\) and~\({ \norm{ u_P^\star }_{ \infty } }\). This will be done in~\cref{sec:subadditivity}; see the discussion following~\Cref{thm:wavespeed-final-bound}.

\section{From the periodic to the solitary-wave problem: a~special minimising sequence} \label{sec:special-min-seq}

As outlined in~\cref{sec:variational-method}, we now construct a special minimising sequence for the solitary-wave problem with help of suitable scalings, truncations and translations of~\({ u_P^\star }\). To this end, we first establish a general asymptotic result as~\({ P \to \infty }\) for convolution operators with integrable kernels.
\begin{lemma} \label{thm:convolution-approximation-periodic-solitary}
Let~\({ f \in \textnormal{L}^1 }\) and~\({ \Set{ \widetilde{ u }_P }_P \subset \textnormal{H}^s }\) be a bounded family of functions with~\({ \support \widetilde{ u }_P \subset \bigl( - \frac{P}{2}, \frac{P}{2} \bigr) }\), and associate, for each~\({ P }\), the periodic extension~\({ u_P \coloneqq \sum_{ j \in \Z } \widetilde{ u }_P(\cdot + jP) \in \textnormal{H}_P^s }\) of~\({ \widetilde{ u }_P }\).
Then
\begin{equation*}
\norm*{ f \ast (\widetilde{ u }_P - u_P) }_{\textnormal{H}^s \left( - \frac{P}{2}, \frac{P}{2}  \right) } \to 0 \qquad \text{and} \qquad \norm{ f \ast \widetilde{ u }_P  }_{ \textnormal{H}^s \left( \Set*{ \abs{ x } > \frac{P}{2} } \right) } \to 0 \qquad \text{as } P \to \infty.
\end{equation*}
\end{lemma}
\begin{proof}
Note first that~\({ f \ast u_P = f_P \ast_P u_P }\) by~\Cref{thm:convolution-periodic}, where~\({ f_P = \sum_{ j \in \Z } f( \cdot + jP) \in \textnormal{L}_P^1 }\). As~such,
\begin{equation} \label{eq:convolution-transfer-to-kernel}
f \ast (\widetilde{ u }_P - u_P)(x) = \int_{ - \frac{P}{2} }^{ \frac{P}{2} } \bigl[ f(x - y) - f_P(x - y) \bigr] \, \widetilde{ u }_P(y) \dee y = (f - f_P) \ast_P \widetilde{ u }_P(x)
\end{equation}
for~\({ x \in \bigl( -\frac{ P }{ 2 }, \frac{P}{2} \bigr)  }\), using that~\({ u_P \equiv \widetilde{ u }_P }\)~there. Young's inequality then gives
\begin{equation*}
\norm*{ f \ast (\widetilde{ u }_P - u_P) }_{\textnormal{L}^2 \left( - \frac{P}{2}, \frac{P}{2}  \right) } \leq \norm{f - f_P}_{ \textnormal{L}^1\left( - \frac{P}{2}, \frac{P}{2} \right)} \norm{ \widetilde{u}_P}_{ \textnormal{L}^2} \xrightarrow[P \to \infty]{} 0,
\end{equation*}
because~\({ \Set{ \widetilde{ u }_P }_P }\) is bounded in~\({ \textnormal{L}^2 }\) and~\({ \norm{f - f_P}_{ \textnormal{L}^1\left( - \frac{P}{2}, \frac{P}{2} \right)} = \norm{f}_{ \textnormal{L}^1\left(\Set*{ \abs{x} > \frac{P}{2}}\right)} \to 0}\) as~\({ P \to \infty }\).

Switching to~\({ \norm{ f \ast \widetilde{ u }_P  }_{ \textnormal{L}^2 \left( \Set*{ \abs{ x } > \frac{P}{2} } \right) } }\), put~\({ v_j \coloneqq f \ast \widetilde{ u }_P( \cdot + jP) }\) and observe from dominated convergence that
\begin{equation*}
\norm{ f \ast \widetilde{ u }_P  }_{ \textnormal{L}^2 \left( \Set*{ \abs{ x } > \frac{P}{2} } \right) }^2 = \sum_{ \abs{j} \geq 1} \int_{- \frac{P}{2}}^{ \frac{P}{2}} \abs{v_j}^2 \dee x = \int_{ - \frac{P}{2} }^{ \frac{P}{2} } \sum_{ \abs{ j } \geq 1 } \abs{ v_j }^2 \dee x \leq \int_{ - \frac{P}{2} }^{ \frac{P}{2} } \abs[\Big]{\sum_{ \abs{ j } \geq 1 } \abs{ v_j } }^2 \dee x,
\end{equation*}
where the last estimate used~\({ \norm{  }_{ \ell^2(\Z \setminus \Set{0}) } \leq \norm{  }_{ \ell^1(\Z \setminus \Set{0}) } }\). Dominated convergence once more yields
\begin{align*}
\sum_{ \abs{ j } \geq 1 } \abs{ v_j(x) } & \leq \sum_{ \abs{ j } \geq 1 } \int_{ - \frac{P}{2} }^{ \frac{P}{2} } \abs{ f( x + jP - y) } \, \abs{ \widetilde{ u }_P(y) } \dee y \\
&= \int_{ - \frac{P}{2} }^{ \frac{P}{2} } \sum_{ \abs{ j } \geq 1 } \abs{f( x + jP - y) } \, \abs{ \widetilde{ u }_P(y) } \dee y \\
&= (\abs{ f }_P - \abs{ f }) \ast_P \abs{ \widetilde{ u }_P }(x),
\end{align*}
for~\({ x \in \bigl( -\frac{ P }{ 2 }, \frac{P}{2} \bigr)  }\), where~\({ \abs{ f }_P \coloneqq \sum_{ j \in \Z } \abs{ f ( \cdot + jP) } }\). Introducing~\({  \abs{ \widetilde{ u }_P }_P \coloneqq  \sum_{ j \in \Z } \abs{ \widetilde{ u }_P( \cdot + jP) } }\) also, we have
\begin{equation*}
(\abs{ f }_P - \abs{ f }) \ast_P \abs{ \widetilde{ u }_P } = \abs{ f } \ast ( \abs{ \widetilde{ u }_P }_P - \abs{ \widetilde{ u }_P })
\end{equation*}
from~\eqref{eq:convolution-transfer-to-kernel}, and so in total,
\begin{equation*}
\norm{ f \ast \widetilde{ u }_P  }_{ \textnormal{L}^2 \left( \Set*{ \abs{ x } > \frac{P}{2} } \right) } \leq \norm[\big]{ \abs{ f } \ast ( \abs{ \widetilde{ u }_P } - \abs{ \widetilde{ u }_P }_P)  }_{\textnormal{L}^2 \left( - \frac{P}{2}, \frac{P}{2} \right) }.
\end{equation*}
Now note that the right-hand side vanishes as~\({ P \to \infty }\) by the first result applied to~\({ \abs{ f } }\) and~\({ \abs{ \widetilde{ u }_P } }\).

With case~\({ s = 0 }\) established, case~\({ s \in \Z_+ }\) follows immediately since convolution commutes with differentiation, and so by interpolation it is true for any~\({ s \geq 0 }\).
\end{proof}

\begin{proposition} \label{thm:approx-periodic-solitary}
Let~\({ \Set{ \widetilde{ u }_P }_P \subset \textnormal{H}^s }\) be a bounded family of functions with~\({ \support \widetilde{ u }_P \subset \bigl( - \frac{P}{2}, \frac{P}{2} \bigr) }\), and define \({ u_P \coloneqq \sum_{ j \in \Z } \widetilde{ u }_P(\cdot + jP) \in \textnormal{H}_P^s }\). Then
\begin{alignat}{4}
\mathcal{A}(\widetilde{ u }_P) - \mathcal{A}_P(u_P) &= 0, \quad \norm{ \mathcal{A}'(\widetilde{ u }_P) - \mathcal{A}_P'(u_P) }_{\textnormal{H}^s \left( - \frac{P}{2}, \frac{P}{2}  \right) } &= 0 \quad &\text{and} \quad \norm{ \mathcal{A}'(\widetilde{ u }_P) }_{ \textnormal{H}^s \left( \Set*{ \abs{ x } > \frac{P}{2} } \right) } = 0 \notag \\
\intertext{for~\({ \mathcal{A} \in \Set{ \mathcal{Q}, \mathcal{N} } }\) and any~\({ P }\), whereas}
\mathcal{L}(\widetilde{ u }_P) - \mathcal{L}_P(u_P) &\to 0, \quad \norm{ \mathcal{L}'(\widetilde{ u }_P) - \mathcal{L}_P'(u_P) }_{\textnormal{H}^s \left( - \frac{P}{2}, \frac{P}{2}  \right) } &\to 0 \quad &\text{and} \quad \norm{ \mathcal{L}'(\widetilde{ u }_P) }_{ \textnormal{H}^s \left( \Set*{ \abs{ x } > \frac{P}{2} } \right) } \to 0 \label{eq:L-approx-periodic-solitary}
\end{alignat}
as~\({ P \to \infty }\). In~particular,~\eqref{eq:L-approx-periodic-solitary} also holds for~\({ \mathcal{E} }\),~\({ \mathcal{E}_P }\).
\end{proposition}
\begin{proof}
Since
\begin{equation*}
n(\widetilde{ u }_P(x)) =
\begin{dcases}
n(u_P(x)) & \text{if } \abs{ x } < \tfrac{ P }{ 2 }; \\
0 & \text{if } \abs{ x } \geq \tfrac{P}{2},
\end{dcases}
\end{equation*}
and similarly for~\({ N }\), we readily obtain the result for~\({ \mathcal{A} = \mathcal{N} }\). Case~\({ \mathcal{A} = \mathcal{Q} }\) is analogous.

As~\({ L }\) is a convolution operator with integrable kernel, \Cref{thm:convolution-approximation-periodic-solitary} gives the last two statements in~\eqref{eq:L-approx-periodic-solitary}. Observe then also that
\begin{equation*}
\abs*{ \mathcal{L}(\widetilde{ u }_P) - \mathcal{L}_P(u_P) } =  \abs[\Bigg]{\tfrac{1}{2} \int_{ - \frac{P}{2} }^{ \frac{P}{2} } \widetilde{ u }_P \left( L \widetilde{ u }_P - L u_P \right)  \dee x } \leq \tfrac{ 1 }{ 2 } \norm{ \widetilde{ u }_P }_{ 0 } \norm{ L \widetilde{ u }_P - L u_P }_{ \textnormal{L}^2 \left( - \frac{P}{2}, \frac{P}{2} \right)  } \xrightarrow[P \to \infty]{} 0.
\end{equation*}
\end{proof}

We now define the \emph{special minimising sequence} for~\({ \mathcal{E} }\) over~\({ U_{ \mu}^s }\) as~follows. Since \({ \norm{ u_P^\star }_{ \textnormal{H}_P^s } \eqsim \mu^{ \frac{ 1 }{ 2 }} }\) holds uniformly over~\({ P \geq P_\mu }\) by \Cref{thm:bound-HsP-norm-mu}, there must---argue by contradiction---be subintervals \({ \Omega_P \coloneqq \left(x_P - \ell_P, x_P + \ell_P \right) }\) of~\({ \bigl( - \frac{P}{2}, \frac{P}{2} \bigr)  }\) such that~\({ \norm{ u_P^\star }_{ \textnormal{H}^s ( \Omega_P) } \to 0 }\) and~\({ \ell_P > 0 }\) satisfies \({ \ell_P / P \to 0 }\) as~\({ P \to \infty }\). We then translate and smoothly truncate~\({ u_P^\star }\) into
\begin{equation} \label{eq:special-min-seq}
\widetilde{ u }_P^{} \coloneqq A_P^{} \chi_P^{} u_P^{\star \, \textnormal{\textsc{t}}} \qquad \text{with }  u_P^{\star \, \textnormal{\textsc{t}}} \coloneqq u_P^\star \left(\cdot + x_P^{} + \tfrac{ P }{ 2 }\right),
\end{equation}
where~\({ \chi_P \in \textnormal{C}_{ \textnormal{c}}^\infty(\R \to [0, 1]) }\) equals
\begin{equation*}
\chi_P(x) =
\begin{dcases}
1 & \text{if } \abs{ x } \leq \tfrac{ P }{ 2 } - \ell_P; \\
0 & \text{if } \abs{ x } \geq \tfrac{ P }{ 2 } - \epsilon,
\end{dcases}
\end{equation*}
for some fixed~\({ \epsilon > 0 }\), and~\({ A_P^{} \coloneqq \sqrt{2 \mu} / \norm{ \chi_P^{} u_P^{\star \, \textnormal{\textsc{t}}}  }_{ 0 } }\), so that
\begin{equation*}
\widetilde{ u }_P \in U_\mu^s  \qquad \text{and} \qquad \support \widetilde{ u }_P \subseteq \Set*{ \abs{ x } \leq \tfrac{P}{2} - \epsilon } \subset \bigl( - \tfrac{P}{2}, \tfrac{P}{2} \bigr).
\end{equation*}
Moreover, let~\({ u_P \coloneqq  \sum_{ j \in \Z } \widetilde{ u }_P( \cdot + jP) \in \textnormal{H}_P^s }\) be the periodisation of~\({ \widetilde{ u }_P }\); see~\cref{fig:min-seq-relationships} for~illustration.

\begin{figure}[h!]
	\vspace*{-.5em}%
    \centering%
    \begin{subfigure}[b]{0.32\textwidth}%
    \centering%
    \includegraphics{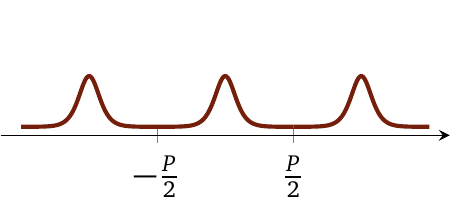}
        \caption{\({ u_P^{\star \, \textnormal{\textsc{t}}} = u_P^\star \left(\cdot + x_P^{} + \tfrac{ P }{ 2 }\right) }\).}
    \end{subfigure}
    ~ 
    \begin{subfigure}[b]{0.32\textwidth}%
    \centering%
    \includegraphics{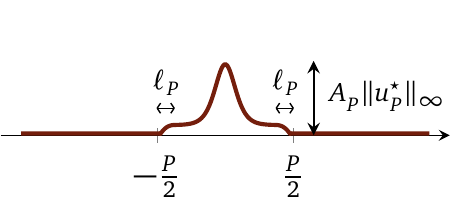}
        \caption{\({ \widetilde{ u }_P^{} }\).}
    \end{subfigure}
    ~ 
    \begin{subfigure}[b]{0.32\textwidth}%
    \centering%
    \includegraphics{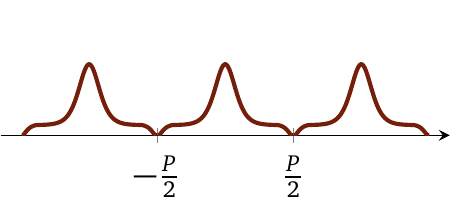}
        \caption{\({ u_P^{} }\).}
    \end{subfigure}
    \caption{Illustrating the relationship between the periodic traveling waves~\({ u_P^\star }\) (real profile unknown), the truncated functions~\({ \widetilde{ u }_P }\) converging to a solitary wave as~\({ P \to \infty }\), and the periodisations~\({ u_P }\) of~\({ \widetilde{ u }_P }\).}\label{fig:min-seq-relationships}
\end{figure}

Intuitively, the more nonlocal \({ L }\)~is---in the sense of \enquote{distributing mass} of~\({ \widetilde{ u }_P }\) from~\({ \bigl( - \frac{P}{2}, \frac{P}{2} \bigr) }\) into its complement---the~faster \({ \ell_P }\) likely should grow, because~\({ \widetilde{ u }_P }\) is asymptotically negligible outside of~\({ \Set*{ \abs{ x } \leq \tfrac{ P }{ 2 } - \ell_P } }\). In~our case, it~suffices in~fact to let~\({ \ell_P  \coloneqq \ell_\star}\) be constant for all~\({ P \geq P_\mu }\). Note~that~\autocite{EGW2012} used~\({ \ell_P \sim P^{\frac{ 1 }{ 4 }} }\).

The \emph{special minimising sequence~\({ \Set{ \widetilde{ u }_k }_{k \in \N} }\)} is now defined as~\({ \widetilde{ u }_k \coloneqq \widetilde{ u }_{P_k} }\), where~\({ \Set{ P_k }_k }\) is an increasing, unbounded sequence with~\({ P_0 \geq P_\mu }\). And in the following results extending~\autocite[Theorem~3.8]{EGW2012}, we show that~\({ \Set{ \widetilde{ u }_k }_k }\) does indeed minimise~\({ \mathcal{E} }\) over~\({ U_\mu^s }\), resembles~\({ u_P^\star }\) with~\({ \norm{ \widetilde{ u }_k^{} }_{ s }^2 \eqsim \mu }\), and approximates the traveling-wave equation~\eqref{eq:traveling-wave} in~\({ \textnormal{H}^s }\). For~convenience, put~\({ \Omega_P^{ \textnormal{\textsc{t}}} \coloneqq \Set*{ \tfrac{ P }{ 2 } - \ell_\star < \abs{ x } <  \tfrac{ P }{ 2 }} }\), so that by~construction,~\({ \norm{ u_P^{\star \, \textnormal{\textsc{t}}} }_{ \textnormal{H}^s(\Omega_P^{ \textnormal{\textsc{t}}}) } \to 0 }\) as~\({ P \to \infty }\).

\begin{lemma} \label{thm:HsP-convergence-special-min-seq}
\({ \norm{ u_P^{} - u_P^{\star \, \textnormal{\textsc{t}}} }_{ \textnormal{H}_P^s } \to 0 }\)\quad and \quad \({ \norm{ \mathcal{E}_P'(u_P^{} )  - \mathcal{E}_P'(u_P^{\star \, \textnormal{\textsc{t}}}) }_{ \textnormal{H}_P^s } \to 0}\)\quad as~\({ P \to \infty }\).
\end{lemma}
\begin{proof}
Since~\({ A_P \to 1 }\), we find that
\begin{equation*}
\norm{ u_P^{} - u_P^{\star \, \textnormal{\textsc{t}}} }_{ \textnormal{L}_P^2 }^2 = \abs{ A_P^{} - 1 }^2 \smashoperator{\int_{ \left( - \frac{P}{2}, \frac{P}{2} \right) \setminus \Omega_P^{ \textnormal{\textsc{t}}} }} \abs{ u_P^{\star \, \textnormal{\textsc{t}}} }^2 \dee x +  \smashoperator{\int_{\vphantom{\left( - \frac{P}{2}, \frac{P}{2} \right)} \Omega_P^{ \textnormal{\textsc{t}}}}} \abs*{ (A_P^{} \chi_P^{} - 1) \, u_P^{\star \, \textnormal{\textsc{t}}} }^2 \dee x \xrightarrow[ P \to \infty ]{  } 0, 
\end{equation*}
because the first integral is less than~\({ \norm{ u_P^{\star \, \textnormal{\textsc{t}}} }_{ \textnormal{L}_P^2 }^2 = 2 \mu }\) whereas the latter is~\({\lesssim \norm{ u_P^{\star \, \textnormal{\textsc{t}}} }_{ \textnormal{L}^2(\Omega_P^{ \textnormal{\textsc{t}}}) }^2 }\), which vanishes. In~a straightforward manner, this extends to~\({ \textnormal{H}_P^s }\) with help of~\eqref{eq:periodicHs-difference-norm}, Leibniz' rule (\({ \lfloor s \rfloor }\)~times) plus the fact that~\({ \norm{ \chi_P^{(i)} }_{ \infty } \lesssim \ell_\star^{-i} \lesssim 1 }\) uniformly in~\({ P }\).

With the first result established, we then find that
\begin{equation*}
\norm{ \mathcal{L}_P'(u_P^{} )  - \mathcal{L}_P'(u_P^{\star \, \textnormal{\textsc{t}}}) }_{ \textnormal{H}_P^s } =  \norm{ L(u_P^{} -  u_P^{\star \, \textnormal{\textsc{t}}})}_{ \textnormal{H}_P^s } \leq \mathfrak{m}(0) \norm{ u_P^{} -  u_P^{\star \, \textnormal{\textsc{t}}}}_{ \textnormal{H}_P^s } \xrightarrow[ P \to \infty ]{  } 0.
\end{equation*}
As regards
\begin{equation*}
\norm{ \mathcal{N}_P'(u_P^{} )  - \mathcal{N}_P'(u_P^{\star \, \textnormal{\textsc{t}}}) }_{ \textnormal{H}_P^s } = \norm{ n(u_P^{} )  - n(u_P^{\star \, \textnormal{\textsc{t}}}) }_{ \textnormal{H}_P^s },
\end{equation*}
observe first that
\begin{equation*}
\norm{ n(u_P^{} )  - n(u_P^{\star \, \textnormal{\textsc{t}}}) }_{ \textnormal{H}^s \left( - \frac{P}{2}, \frac{P}{2} \right) \setminus \Omega_P^{ \textnormal{\textsc{t}}} } \xrightarrow[ P \to \infty ]{  } 0,
\end{equation*}
essentially because~\({ A_P \to 1 }\). Specifically, one may argue by the chain rule and dominated convergence---a~linear combination of~\({ n(u_P^{\star \, \textnormal{\textsc{t}}}) }\) and its \({ \lfloor s \rfloor }\)~derivatives, all of which are uniformly bounded in~\({ \textnormal{L}_P^2 }\), serves as a dominating~function---because~\({ \left(u_P^{} - u_P^{\star \, \textnormal{\textsc{t}}} \right) 1_{\left( - \frac{P}{2}, \frac{P}{2} \right) \setminus \Omega_P^{ \textnormal{\textsc{t}}}} }\) and its \({ \lfloor s \rfloor }\)~derivatives converge pointwise to~\({ 0 }\)~a.e.\@ as~\({ P \to \infty }\), and hence, also
\begin{equation*}
\frac{\dee{}^i  }{\dee x^i}  \left[n(u_P(x)) - n(u_P^{\star \, \textnormal{\textsc{t}}}(x)) \right] 1_{\left( - \frac{P}{2}, \frac{P}{2} \right) \setminus \Omega_P^{ \textnormal{\textsc{t}}}}(x) \xrightarrow[ P \to \infty ]{ \textnormal{a.e.} } 0
\end{equation*}
for all~\({i = 0, \dotsc, \lfloor s \rfloor }\). Moreover,
\begin{equation*}
 \norm{ n(u_P^{} )  - n(u_P^{\star \, \textnormal{\textsc{t}}}) }_{ \textnormal{H}^s ( \Omega_P^{ \textnormal{\textsc{t}}}) } \leq \norm{ n(u_P) }_{ \textnormal{H}^s ( \Omega_P^{ \textnormal{\textsc{t}}})  } + \norm{ n(u_P^{\star \, \textnormal{\textsc{t}}}) }_{ \textnormal{H}^s ( \Omega_P^{ \textnormal{\textsc{t}}}) }.
\end{equation*}
On the right-hand side, the first term is controlled by the latter, rigorously due to Leibniz' rule and~\({ A_P }\) being bounded. And, arguing similarly as~\eqref{eq:chain-rule-HsP-proof}, we also have
\begin{equation*}
\norm{ n(u_P^{\star \, \textnormal{\textsc{t}}}) }_{ \textnormal{H}^s ( \Omega_P^{ \textnormal{\textsc{t}}}) } \lesssim \norm{ u_P^{\star \, \textnormal{\textsc{t}}} }_{ \textnormal{L}^{ \infty}( \Omega_P^{ \textnormal{\textsc{t}}}) }^q \norm{ u_P^{\star \, \textnormal{\textsc{t}}} }_{ \textnormal{H}^s ( \Omega_P^{ \textnormal{\textsc{t}}}) } \xrightarrow[ P \to \infty ]{  } 0,
\end{equation*}
with \({ \norm{ u_P^{\star \, \textnormal{\textsc{t}}} }_{ \textnormal{L}^{ \infty}( \Omega_P^{ \textnormal{\textsc{t}}}) } \leq \norm{ u_P^{\star} }_{ \infty } \lesssim \norm{ u_P^{\star} }_{ \textnormal{H}_P^s} < R  }\). Hence, \({ \norm{ \mathcal{N}_P'(u_P^{} )  - \mathcal{N}_P'(u_P^{\star \, \textnormal{\textsc{t}}}) }_{ \textnormal{H}_P^s } \to 0 }\) as \({ P \to \infty }\), and the proof is complete.
\end{proof}

\begin{proposition} \label{thm:convergence-periodic-solitary-infimum}
\({ \Set{ \widetilde{ u }_k }_k }\) is a minimising sequence for~\({ \mathcal{E} }\) over~\({ U_\mu^s }\), and
\begin{equation*}
I_{P, \mu} \xrightarrow[ P \to \infty ]{  } I_\mu,
\end{equation*}
where~\({ I_{P, \mu}^{} \coloneqq \mathcal{E}_P^{}(u_P^{\star \, \textnormal{\textsc{t}}}) }\) is the minimum of the periodic problem.
\end{proposition}
\begin{proof}
Writing~\({ \mathcal{E}(\widetilde{ u }_P^{}) = \left( \mathcal{E}(\widetilde{ u }_P^{}) - \mathcal{E}_P^{}(u_P^{}) \right) + \left( \mathcal{E}_P^{}(u_P^{}) - \mathcal{E}_P^{}(u_P^{\star \, \textnormal{\textsc{t}}}) \right) + I_{P, \mu}^{} }\) and observing by~\Cref{thm:approx-periodic-solitary} and \Cref{thm:HsP-convergence-special-min-seq} that
\begin{equation*}
\mathcal{E}(\widetilde{ u }_P^{}) - \mathcal{E}_P^{}(u_P^{}) \xrightarrow[ P \to \infty ]{  } 0
\end{equation*}
and
\begin{equation*}
\mathcal{E}_P^{}(u_P^{}) - \mathcal{E}_P^{}(u_P^{\star \, \textnormal{\textsc{t}}}) \leq \sup_{ u \in U_{P, \mu}^s } \norm{ \mathcal{E}_P'(u) }_{ \textnormal{L}_P^2 } \norm{ u_P^{} - u_P^{\star \, \textnormal{\textsc{t}}} }_{ \textnormal{L}_P^2 }  \xrightarrow[ P \to \infty ]{  } 0,
\end{equation*}
we get
\begin{equation*}
I_{ \mu} \leq \liminf_{ P \to \infty } \mathcal{E}( \widetilde{ u }_P) = \liminf_{ P \to \infty } I_{P, \mu}.
\end{equation*}
Here we used that \({ \norm{ \mathcal{E}_P'(u) }_{ \textnormal{L}_P^2 } }\) is uniformly bounded over~\({ u \in U_{P, \mu}^s }\), since \({ \norm{ \mathcal{L}_P'(u) }_{ \textnormal{L}_P^2 } \leq \mathfrak{m}(0) \norm{ u }_{ \textnormal{L}_P^2 } \lesssim \mu  }\) and
\begin{equation*}
\norm{ \mathcal{N}_P'(u) }_{ \textnormal{L}_P^2 } = \norm{ n(u) }_{ \textnormal{L}_P^2 } \lesssim \norm{ u }_{ \textnormal{L}_P^2 } 
\begin{cases}
\mu^{ \theta q}  & \text{if } s \leq \tfrac{ 1 }{ 2 }; \\
\norm{ u }_{ \infty }^q & \text{if } s > \tfrac{ 1 }{ 2 },
\end{cases}
\end{equation*}
with~\({ \norm{ u }_{ \infty } \lesssim \norm{ u }_{ \textnormal{H}_P^s } < R }\).

Conversely, let~\({ \widetilde{ w } \in \textnormal{C}_{ \textnormal{c}}^{ \infty} }\) satisfy~\({ \mathcal{Q}( \widetilde{ w }) = \mu }\), and put~\({ w_P \coloneqq \sum_{ j \in \Z } \widetilde{ w }( \cdot + jP) }\), so that~\({ I_{P, \mu} \leq \mathcal{E}_P(w_P) }\) and~\({ \mathcal{E}_P(w_P) \to \mathcal{E}( \widetilde{ w }) }\) as~\({ P \to \infty }\) by~\Cref{thm:approx-periodic-solitary}. Then
\begin{equation*}
\limsup_{ P \to \infty } I_{P, \mu} \leq \mathcal{E}(\widetilde{ w }),
\end{equation*}
and consequently also
\begin{equation*}
\limsup_{ P \to \infty } I_{P, \mu} \leq  \inf \Set*{ \mathcal{E}(u) \given u \in \textnormal{C}_{ \textnormal{c}}^{ \infty} \cap U_\mu^s } = I_{ \mu}
\end{equation*}
by continuity of~\({ \mathcal{E} }\) and density.
\end{proof}
\begin{proposition} \label{thm:special-min-seq-for-E}
The special minimising sequence~\({ \Set{ \widetilde{ u }_k }_k }\) satisfies
\begin{equation*}
\sup\nolimits_{ k }^{} \norm{ \widetilde{ u }_k^{} }_{ s } \eqsim  \mu^{ \frac{ 1 }{ 2 }} \qquad \text{and} \qquad \norm{ \mathcal{E}'(\widetilde{ u }_k ) + \nu_k \mathcal{Q}'(\widetilde{ u }_k)}_{ s } \xrightarrow[ k \to \infty ]{  } 0, 
\end{equation*}
where~\({\!\nu_k \coloneqq \nu_{P_k} }\). In~fact, we may assume that~\({ \!\nu_k }\) does not depend on~\({ k }\).
\end{proposition}
\begin{proof}
\Cref{thm:existence-periodic} and \Cref{thm:HsP-convergence-special-min-seq} directly imply
\begin{equation*}
\norm{ \widetilde{ u }_P^{} }_{ s } \eqsim \norm{ u_P^{} }_{ \textnormal{H}_P^s } \leq \norm{ u_P^{} - u_P^{\star \, \textnormal{\textsc{t}}} }_{ \textnormal{H}_P^s }  + \norm{ u_P^{\star \, \textnormal{\textsc{t}}} }_{ \textnormal{H}_P^s } \lesssim \mu^{ \frac{1}{2}}
\end{equation*}
for all~\({ P \geq P_\mu }\), where~\({ P_\mu }\) is replaced by a larger constant if~necessary. Furthermore,
\begin{equation*}
\norm*{ \mathcal{E}'(\widetilde{ u }_P^{} ) + \nu_P^{} \mathcal{Q}'(\widetilde{ u }_P^{})}_{ s } \leq \underbrace{\norm*{ \mathcal{E}'(\widetilde{ u }_P^{} ) + \nu_P^{} \mathcal{Q}'(\widetilde{ u }_P^{})}_{ \textnormal{H}^s \left( - \frac{P}{2}, \frac{P}{2} \right) }}_{ \eqqcolon \mathfrak{I}_1} + \underbrace{\norm*{ \mathcal{E}'(\widetilde{ u }_P^{}) + \nu_P^{} \mathcal{Q}'(\widetilde{ u }_P^{})}_{ \textnormal{H}^s \left( \Set*{ \abs{ x } > \frac{P}{2} } \right) }}_{ \eqqcolon \mathfrak{I}_2},
\end{equation*}
where
\begin{subequations}
\begin{align} \SwapAboveDisplaySkip
\mathfrak{I}_1 &\leq \norm*{ \mathcal{E}'(\widetilde{ u }_P^{} )  - \mathcal{E}_P'(u_P^{}) }_{ \textnormal{H}^s \left( - \frac{P}{2}, \frac{P}{2} \right) } + \nu_P^{} \underbrace{\norm*{ \mathcal{Q}'(\widetilde{ u }_P^{} )  - \mathcal{Q}_P'(u_P^{}) }_{ \textnormal{H}^s \left( - \frac{P}{2}, \frac{P}{2} \right) }}_{= 0} \label{eq:I1-line1} \\[1ex]
&\hphantom{\leq} + \norm*{ \mathcal{E}_P'(u_P^{} )  - \mathcal{E}_P'(u_P^{\star \, \textnormal{\textsc{t}}}) }_{ \textnormal{H}^s \left( - \frac{P}{2}, \frac{P}{2} \right) } + \nu_P^{} \norm*{ \mathcal{Q}'(u_P^{} )  - \mathcal{Q}_P'(u_P^{\star \, \textnormal{\textsc{t}}}) }_{ \textnormal{H}^s \left( - \frac{P}{2}, \frac{P}{2} \right) } \label{eq:I1-line2} \\[1.2ex]
&\hphantom{\leq} + \underbrace{\norm*{ \mathcal{E}_P'(u_P^{\star \, \textnormal{\textsc{t}}} ) + \nu_P^{} \mathcal{Q}_P'(u_P^{\star \, \textnormal{\textsc{t}}}) }_{ \textnormal{H}^s \left( - \frac{P}{2}, \frac{P}{2} \right) }}_{= 0}, \label{eq:I1-line3}
\end{align}
\end{subequations}
vanishes as~\({ P \to \infty }\) due to~\Cref{thm:approx-periodic-solitary} for~\eqref{eq:I1-line1}; \Cref{thm:HsP-convergence-special-min-seq} plus the fact that~\({ \!\nu_P^{} \mathcal{Q}_P' }\) is a continuous linear operator on~\({ \textnormal{H}_P^s }\)---using that~\({ \Set{ \nu_P }_P }\) is bounded---for~\eqref{eq:I1-line2}; \({ u_P^{\star \, \textnormal{\textsc{t}}} }\)~solving~\eqref{eq:traveling-wave} in~\({ \textnormal{H}_P^s }\) for~\eqref{eq:I1-line3},
and
\begin{equation*} 
\mathfrak{I}_2 = \norm*{ \mathcal{E}'(\widetilde{ u }_P^{}) }_{ \textnormal{H}^s \left( \Set*{ \abs{ x } > \frac{P}{2} } \right) } + \nu_P^{} \underbrace{\norm*{ \mathcal{Q}'(\widetilde{ u }_P^{}) }_{ \textnormal{H}^s \left( \Set*{ \abs{ x } > \frac{P}{2} } \right) }}_{= 0}
\end{equation*}
vanishes by~\Cref{thm:approx-periodic-solitary}.

Finally, since~\({ \Set{ \nu_k }_k }\) is bounded, it admits a convergent subsequence, and we therefore conclude, noting that~\({ \norm{ \mathcal{Q}'(\widetilde{ u }_k) }_{ s } = \norm{ \widetilde{ u }_k }_{ s }  }\) is uniformly~bounded in~\({ k }\).
\end{proof}

\section{Strict subadditivity and bounds in~\({ \textnormal{L}^\infty }\) and for the wave speed} \label{sec:subadditivity}

In this section we establish that~\({ \mu \mapsto I_\mu }\) is strictly subadditive~\eqref{eq:strictly-subadditive} on some interval~\({ (0, \mu_\star) }\) in order to rule out the case of dichotomy in Lion's principle, see~\cref{sec:concentration-compactness}, and along the way also obtain improved lower bounds for the wave speed and upper bounds in~\({ \textnormal{L}^\infty }\). In~fact, we prove that~\({ \mu \mapsto I_\mu }\) is \emph{strictly subhomogeneous} on~\({ (0, \mu_\star) }\), meaning that
\begin{equation} \label{eq:subhomogeneity}
I_{a \mu} < a I_\mu \qquad \text{whenever } 0 < \mu < a \mu < \mu_\star,
\end{equation}
which in turn implies strict subadditivity:
\begin{equation*}
I_{ \mu_1 + \mu_2} < \left( \tfrac{ \mu_1}{ \mu_2} + 1 \right) I_{ \mu_2} = \tfrac{ \mu_1}{ \mu_2} I_{ \frac{ \mu_2 }{ \mu_1} \cdot \mu_1} + I_{ \mu_2 } \leq I_{ \mu_1} + I_{ \mu_2}.
\end{equation*}

Observe that if the nonlinearity~\({ n }\) is homogeneous, then~\eqref{eq:subhomogeneity} follows directly from a scaling argument because~\({ \mathcal{E} }\) is homogeneous. In~the presence of~\({ \mathcal{N}_\textnormal{r} }\), however, we need that~\({ \mathcal{N}_\textnormal{r}(u) = o( \mu^{q \alpha}) }\). This would be guaranteed provided
\begin{equation} \label{eq:Linfty-improved-bound}
\norm{ u }_{ \infty } \lesssim \mu^\alpha
\end{equation}
holds uniformly for a minimising sequence, which as we shall see, is the case for the special minimising sequence~\({ \Set{ \widetilde{ u }_k }_k }\) in~\cref{sec:special-min-seq}.

As a first step toward~\eqref{eq:subhomogeneity} and~\eqref{eq:Linfty-improved-bound}, we require a \({ \mu }\)-dependent upper bound on~\({ I_\mu }\). Following~\autocite{EGW2012}, it seems natural to introduce the homogeneous, long-wave part~\({ \mathcal{E}_{ \textnormal{lw}} \coloneqq \mathcal{L}_{ \textnormal{lw}} + \mathcal{N}_q }\) of~\({ \mathcal{E} }\), where
\begin{equation*}
\mathcal{L}_{ \textnormal{lw}}(u) \coloneqq - \frac{\mathfrak{m}^{(2\ell)}(0)}{(2\ell)!} \int_{ \R } \abs{ u^{(\ell)}}^2 \dee x,
\end{equation*}
and consider scalings~\({ S_{ \textnormal{lw}}u \coloneqq \mu^\alpha u( \mu^\beta \cdot) }\) with~\({ \alpha, \beta > 0 }\). We~must have \({ 2 \alpha - \beta = 1 }\) in order for~\({ S_{ \textnormal{lw}} }\) to map~\({ U_1^s }\) into~\({ U_\mu^s }\) (for~\({ \mu }\) sufficiently small), whereas the condition \({ 2 \alpha + (2 \ell - 1) \beta = (2 + q) \alpha - \beta }\) arises naturally in balancing dispersion and nonlinear effects---that~is, \({ \mathcal{L}_{ \textnormal{lw}} }\) and~\({ \mathcal{N}_q }\). This yields
\begin{equation*}
\alpha = \frac{2 \ell}{4 \ell - q} \qquad \text{and} \qquad \beta = \frac{q}{4 \ell - q}.
\end{equation*}
If~\({ u \in U_1^{\ell + 1} }\), then a routine calculation using the scaling properties of~\({ \mathscr{F} }\) gives
\begin{equation} \label{eq:Energy-lw-scaling}
\mathcal{E}(S_{ \textnormal{lw}}u) +  \mathfrak{m}(0) \mu = \mu^{1 + q \alpha} \mathcal{E}_{ \textnormal{lw}}(u) + o(\mu^{1 + q \alpha}),
\end{equation}
noting that the last term encaptures the effects of~\({ \mathcal{N}_\textnormal{r} }\) and the Taylor remainder of~\({ \mathfrak{m} }\). Note that when~\({ s \leq \frac{ 1 }{ 2 } }\), we implicitly choose~\({ \mu }\) so small that~\({ S_{ \textnormal{lw}}u }\) does not see the cut-off~\eqref{eq:n-cutoff} in~\({ n }\)---this works because~\({ \alpha > \frac{ 1 }{ 2 } > \theta }\), where~\({ \theta }\) is as in~\eqref{eq:n-cutoff}. Almost verbatim from~\autocite[Corollary~3.4]{EGW2012}, we now obtain the following.

\begin{lemma} \label{thm:solitary-inf-upper-bound}
There exists a constant~\({ I_\star > 0 }\) such that, for all sufficiently small~\({ \mu }\),
\begin{align} 
I_\mu &< - \mathfrak{m}(0) \mu - I_\star \mu^{1 + q \alpha}  \label{eq:E-inf-estimate} \\
\shortintertext{and, uniformly over~\({ P \geq P_\mu }\),}
I_{P, \mu} &< - \mathfrak{m}(0) \mu - I_\star \mu^{1 + q \alpha}.  \label{eq:E-per-inf-estimate}
\end{align}
\end{lemma}
\begin{proof}
Take any~\({ \varphi \in \textnormal{C}_{ \textnormal{c}}^\infty }\) with~\({ \mathcal{Q}( \varphi) = 1 }\) and define~\({ u = \sqrt{ \lambda} \varphi( \lambda \cdot) }\). Then
\begin{equation*}
\mathcal{E}_{ \textnormal{lw}}(u) = \lambda^{2 \ell} \mathcal{L}_{ \textnormal{lw}} (\varphi) + \lambda^{q / 2} \mathcal{N}_q( \varphi) < 0
\end{equation*}
for all sufficiently small~\({ \lambda }\) provided that~\({ q < 4 \ell }\) and~\({ \mathcal{N}_q( \varphi) < 0 }\), the latter of which holds under Assumption~\ref{assumption:nonlinearity} by choosing~\({ \varphi > 0 }\) if~\({ \gamma > 0 }\) and~\({ \varphi < 0 }\) if~\({ \gamma < 0 }\). Utilising~\eqref{eq:Energy-lw-scaling} and~\Cref{thm:convergence-periodic-solitary-infimum}, this establishes both~\eqref{eq:E-inf-estimate} and~\eqref{eq:E-per-inf-estimate} for sufficiently small~\({ \mu }\) and large~\({ P_\mu }\) with~\({ I_\star = - \tfrac{ 1 }{ 4 }\mathcal{E}_{ \textnormal{lw}}(u) }\),~say.
\end{proof}

With \Cref{thm:special-min-seq-for-E} and \Cref{thm:solitary-inf-upper-bound} at hand, we now restrict our attention to \enquote{special near-minimisers} \({ u \in U_\mu^s \cap \textnormal{L}^\infty }\) of~\({ \mathcal{E} }\) satisfying
\begin{equation} \label{eq:special-near-min}
\mathcal{E}(u) <  -\mathfrak{m}(0) \mu - I_\star \mu^{1 + q \alpha} \qquad \text{and} \qquad \norm{ \mathcal{E}'(u) + \nu \mathcal{Q}'(u) }_{ \textnormal{H}^s \cap \textnormal{L}^\infty } \lesssim \mu^M
\end{equation}
for some~\({ \nu \in \R }\) and large number~\({ M \geq \max \Set*{ \frac{1}{2} + q \alpha,  \frac{1}{2}(1 - q)^{-1} } }\) (with the last term present only when~\({ q < 1 }\)). Here \({ \norm{  }_{ \textnormal{H}^s \cap \textnormal{L}^\infty } \coloneqq \norm{  }_{ s } + \norm{  }_{ \infty }  }\). In~close analogy to \Cref{thm:lower-bound-wave-speed-periodic,thm:bound-HsP-norm-mu}, with help of the identity
\begin{equation} \label{eq:special-identity}
\nu u - L u = n(u) + \mathcal{E}'(u) + \nu \mathcal{Q}'(u),
\end{equation}
one obtains the following result.

\begin{proposition} \label{thm:wave-speed-lower-bound-solitary}
The estimates~\({ \norm{ u }_{ s } \eqsim \mu^{ \frac{ 1 }{ 2 }} }\) and
\begin{equation} \label{eq:wave-speed-lower-bound-solitary}
\nu - \mathfrak{m}(0) \gtrsim \mu^{q \alpha} + o\bigl( \norm{ u }_{ \infty }^q\bigr)
\end{equation}
hold uniformly over the set of special near minimisers~\eqref{eq:special-near-min}.
\end{proposition}

Next we decompose~\({ u }\) into its low and high-frequency components~\({ u_{ \textnormal{lo}} }\) and~\({ u_{ \textnormal{hi}} }\), so that~\({ u_{ \textnormal{lo}} }\) picks up the KdV-type behaviour of~\({ \mathfrak{m} }\) around~\({ 0 }\) and the operator~\({ \!\nu - L }\) may be inverted in~\({ \textnormal{H}^s }\) with regards to~\({ u_{ \textnormal{hi}} }\). Specifically, choose~\({ \xi_0 > 0 }\) in the interval around~\({ 0 }\) where the expansion of~\({ \mathfrak{m} }\) in Assumption~\ref{assumption:dispersion}~\ref{assumption:dispersion-expansion} holds such that~\({ \mathfrak{m}(\xi) \leq \tau \mathfrak{m}(0) }\) for~\({ \abs{ \xi }\geq \xi_0 - \delta }\), where~\({ \tau \in (0, 1) }\) and~\({ 0 < \delta \ll \xi_0 }\), and define operators~\({f \mapsto f_{ \textnormal{lo}} }\) and~\({f \mapsto f_{ \textnormal{hi}} }\) by
\begin{equation} \label{eq:freq-decomp}
\widehat{  f_{ \textnormal{lo}}  } = \varphi \,  \widehat{  f  } \qquad \text{and} \qquad \widehat{  f_{ \textnormal{hi}}  } = (1 - \varphi) \, \widehat{  f  },
\end{equation}
where~\({ \varphi \in \textnormal{C}_{ \textnormal{c}}^\infty(\R \to [0, 1]) }\) equals~\({ 1 }\) for~\({ \abs{ \xi } \leq \xi_0 - \delta }\) and~\({ 0 }\) for~\({ \abs{ \xi } \geq \xi_0 }\). Now~\eqref{eq:special-identity} splits into
\begin{alignat}{2}
(\nu - L) u_{ \textnormal{lo}} &= n(u)_{ \textnormal{lo}} + {}&& (\mathcal{E}'(u) + \nu \mathcal{Q}'(u) )_{ \textnormal{lo}}  \label{eq:low-frequency-comp}\\
(\nu - L) u_{ \textnormal{hi}} &= n(u)_{ \textnormal{hi}} + {}&& (\mathcal{E}'(u) + \nu \mathcal{Q}'(u) )_{ \textnormal{hi}} , \label{eq:high-frequency-comp}
\end{alignat}
and this helps us to establish~\eqref{eq:Linfty-improved-bound}.

\begin{proposition} \label{thm:bound-Hs-norm-mu}
The estimate
\begin{equation*}
\norm{ u }_{ \infty } \lesssim \mu^\alpha
\end{equation*}
holds uniformly over the set of special near minimisers~\eqref{eq:special-near-min}.
\end{proposition}
\begin{proof}
Suppose first that the high-frequency component dominates in~\({ \textnormal{L}^\infty }\), that is, \({ \norm{ u_{ \textnormal{hi}} }_{ \infty } \geq \norm{ u_{ \textnormal{lo}} }_{ \infty } }\), so that in particular, \({ \norm{ u }_{ \infty } \lesssim \norm{ u_{ \textnormal{hi}} }_{ \infty }  }\). When~\({ s \leq \frac{ 1 }{ 2 } }\), it is not clear \textit{a~priori} that~\({ \norm{ u_{ \textnormal{hi}} }_{ \infty } \lesssim \norm{ u_{ \textnormal{hi}} }_{ s }  }\). It~turns out to be almost true, as can be seen as follows. Young's inequality gives \newpage 
\begin{equation*}
\norm{ n(u)_{ \textnormal{lo}} }_{ \infty } = \norm{ (\mathscr{F}^{-1}\varphi) \ast n(u) }_{ \infty } \leq \norm{ \mathscr{F}^{-1}\varphi }_{ \textnormal{L}^1 } \norm{ n(u) }_{ \infty } \lesssim \norm{ n(u) }_{ \infty }, 
\end{equation*}
and likewise
\begin{equation*}
\norm{ (\mathcal{E}'(u) + \nu \mathcal{Q}'(u) )_{ \textnormal{lo}} }_{ \infty } \lesssim  \norm{ \mathcal{E}'(u) + \nu \mathcal{Q}'(u) }_{ \infty } \lesssim \mu^M.
\end{equation*}
Hence
\begin{align*} \SwapAboveDisplaySkip
\norm{ n(u)_{ \textnormal{hi}} }_{ \infty } &\leq \norm{ n(u) }_{ \infty } + \norm{ n(u)_{ \textnormal{lo}} }_{ \infty } \\
&\lesssim \norm{ n(u) }_{ \infty } \\ 
&\lesssim \mu^{ \theta q} \norm{ u }_{ \infty } \\
&\lesssim \mu^{ \theta q} \norm{ u_{ \textnormal{hi}} }_{ \infty },
\end{align*}
using~\eqref{eq:n-cutoff-estimate}, and similarly
\begin{equation*}
\norm{ (\mathcal{E}'(u) + \nu \mathcal{Q}'(u) )_{ \textnormal{hi}} }_{ \infty } \lesssim \mu^M.
\end{equation*}
We find from~\eqref{eq:high-frequency-comp} that
\begin{align*} \SwapAboveDisplaySkip
(\nu - \mu^{ \theta q}) \norm{ u_{ \textnormal{hi}} }_{ \infty } &\lesssim \norm{ Lu_{ \textnormal{hi}} }_{ \infty } + \mu^M \\
&\lesssim \norm{ Lu_{ \textnormal{hi}} }_{ s + \abs{ \sigma } } + \mu^M \\
&\lesssim \norm{ u_{ \textnormal{hi}} }_{ s } + \mu^M,
\end{align*}
and so for~\({ \mu }\) small enough it follows that~\({ \norm{ u_{ \textnormal{hi}} }_{ \infty } \lesssim \norm{ u_{ \textnormal{hi}} }_{ s } + \mu^M }\) when~\({ s \leq \frac{ 1 }{ 2 } }\).

\enlargethispage{.5\baselineskip} 

\Cref{thm:wave-speed-lower-bound-solitary} next implies that~\({\!\nu \geq (\tau + \epsilon) \mathfrak{m}(0) }\) for all sufficiently small \({ \epsilon > 0 }\) and~\({ \mu }\). Hence \({\!\nu - \mathfrak{m}(\xi) \geq \epsilon \mathfrak{m}(0) }\) on~\({ \Set{ \abs{ \xi }\geq \xi_0 - \delta } \supseteq \support (1 - \varphi) }\), which means that the linear operator
\begin{equation*}
\mathscr{F}^{-1} \left[ (\nu - \mathfrak{m})^{-1} (1 - \varphi) \mathscr{F} \right] \colon \textnormal{H}^s \to \textnormal{H}^s
\end{equation*}
is uniformly bounded in norm over~\({\!\nu \geq  (\tau + \epsilon) \mathfrak{m}(0) }\). Consequently, \eqref{eq:high-frequency-comp} and the fractional chain rule~\eqref{eq:chain-rule-Hs} yield
\begin{align*} \SwapAboveDisplaySkip
\norm{ u_{ \textnormal{hi}} }_{ s } &\lesssim \norm{ n(u) }_{ s } + \norm{ \mathcal{E}'(u) + \nu \mathcal{Q}'(u) }_{ s } \\
&\lesssim \norm{ u }_{ s } \norm{ u }_{ \infty }^{ q }  + \mu^M \\
&\lesssim \mu^{ \frac{ 1 }{ 2 }} \norm{ u_{ \textnormal{hi}} }_{ \infty }^q + \mu^M,
\end{align*}
and therefore also
\begin{equation} \label{eq:high-freq-comp-geometric}
\norm{ u_{ \textnormal{hi}} }_{ \infty } \lesssim \mu^{ \frac{1}{2}} \norm{ u_{ \textnormal{hi}} }_{ \infty }^{ q } + \mu^M.
\end{equation}
Now note that
\begin{equation*}
\norm{ u_{ \textnormal{hi}} }_{ \infty } \lesssim \norm{ u_{ \textnormal{hi}} }_{ s } + \mu^M \lesssim \mu^{ \frac{ 1 }{ 2 }} + \mu^M.
\end{equation*}
If~\({ q \geq  1 }\), then~\eqref{eq:high-freq-comp-geometric} shows that~\({ \norm{ u_{ \textnormal{hi}}}_{ \infty } \lesssim \mu^M = o(\mu^\alpha) }\) for sufficiently small~\({ \mu }\). If~\({ q < 1 }\), then~\eqref{eq:high-freq-comp-geometric} yields
\begin{equation*}
\norm{ u_{ \textnormal{hi}} }_{ \infty } \lesssim \mu^{ \frac{1}{2}(1 - q)^{-1}} + \mu^{M} = o(\mu^\alpha)
\end{equation*}
for sufficiently small~\({ \mu }\) due to~\({ M \geq \frac{1}{2}(1 - q)^{-1} > \alpha }\) and the fact that
\begin{equation} \label{eq:elementary-inequality}
x \leq a x^c + b \qquad \text{implies} \qquad x \lesssim a^{(1 - c)^{-1}} + b
\end{equation}
for~\({ x, a, b > 0 }\) and~\({ c \in (0, 1) }\). (To get~\eqref{eq:elementary-inequality}, note first that~\({ x \leq 2 \max \{ ax^c, b \} }\). If~\({ b }\) is the maximum, then \({ x \lesssim b \leq a^{(1 - c)^{-1}} + b }\). Otherwise, \({ x \leq (2a)^{(1 - c)^{-1}} }\), which gives \({ x \lesssim a^{ (1 - c)^{-1}} \leq a^{ (1 - c)^{-1}} + b }\).)

Suppose instead that the low-frequency component dominates: \({ \norm{ u_{ \textnormal{lo}} }_{ \infty } > \norm{ u_{ \textnormal{hi}} }_{ \infty } }\). By Maclaurin expansion of~\({ \mathfrak{m} }\) and~\eqref{eq:wave-speed-lower-bound-solitary} we have
\begin{equation} \label{eq:wavespeed-nearzero-inf-bnd}
\nu - \mathfrak{m}(\xi) > \nu - \mathfrak{m}(0) - \frac{ c \mathfrak{m}^{(2\ell)}(0)}{(2 \ell)!} \xi^{2 \ell} > - \frac{ c \mathfrak{m}^{(2\ell)}(0)}{(2 \ell)!} \xi^{2 \ell} + o\bigl( \norm{ u }_{ \infty }^q\bigr)
\end{equation}
for some~\({ c > 0 }\) when~\({ \abs{ \xi } < \xi_0 }\). Thus
\begin{equation} \label{eq:optimal-inf-estimate1}
\norm{ u_{ \textnormal{lo}}^{(2\ell)} }_{ 0 } \lesssim \norm{ (\nu - \mathfrak{m}) \widehat{  u_{ \textnormal{lo}}  }  }_{ 0 } + \norm{ u_{ \textnormal{lo}} }_{ 0 } \, o\bigl(\norm{ u }_{ \infty }^{ q }\bigr).
\end{equation}
Equation~\eqref{eq:low-frequency-comp} further gives
\begin{align}  \label{eq:optimal-inf-estimate2}
\begin{aligned}
\norm{ (\nu - \mathfrak{m}) \widehat{  u_{ \textnormal{lo}}  }  }_{ 0 }  &\lesssim \norm{ n(u) }_{ 0 } + \norm{ \mathcal{E}'(u) + \nu \mathcal{Q}'(u) }_{ 0 } \\
&\lesssim \mu^{ \frac{1}{2}} \norm{ u }_{ \infty }^{ q } + \mu^M \\
&\lesssim \mu^{ \frac{1}{2}} \norm{ u_{ \textnormal{lo}} }_{ \infty }^{ q } + \mu^M,
\end{aligned}
\end{align}
and so we obtain
\begin{equation} \label{eq:optimal-inf-estimate4}
\norm{ u_{ \textnormal{lo}}^{(2\ell)} }_{ 0 } \lesssim \mu^{ \frac{1}{2}} \norm{ u_{ \textnormal{lo}} }_{ \infty }^{ q } + \mu^M.
\end{equation}
Gagliardo--Nirenberg's inequality then shows that
\begin{equation*}
\norm{ u_{ \textnormal{lo}} }_{ \infty } \lesssim \norm{ u_{ \textnormal{lo}} }_{ 0 }^{1 - \frac{1}{4\ell}} \norm{ u_{ \textnormal{lo}}^{(2\ell)} }_{ 0 }^{ \frac{1}{4\ell} } \lesssim \mu^{ \frac{1}{2}} \norm{ u_{ \textnormal{lo}} }_{ \infty }^{\frac{q}{4\ell}} + \mu^{ \widetilde{ M }},
\end{equation*}
where~\({ \widetilde{ M } \coloneqq \frac{1}{2} \left( 1 - \frac{1}{4 \ell} \right) + \frac{M}{4 \ell} \geq \alpha  }\), from which we finally deduce that
\begin{equation*}
\norm{ u_{ \textnormal{lo}} }_{ \infty } \lesssim  \mu^{ \frac{1}{2}\left(1 -  \frac{q}{4\ell} \right)^{-1} } + \mu^{ \widetilde{ M }} = \mu^\alpha + \mu^{\widetilde{ M } } \lesssim \mu^{ \alpha}
\end{equation*}
with help of~\eqref{eq:elementary-inequality} for~\({ c = q / 4\ell  }\).
\end{proof}
\vspace*{-1em} 
\enlargethispage{.8\baselineskip} 
\begin{remark}
Note that the estimates obtained in the case~\({ \norm{ u_{ \textnormal{hi}} }_{ \infty } \geq \norm{ u_{ \textnormal{lo}} }_{ \infty } }\) in the proof of \Cref{thm:bound-Hs-norm-mu} are (slightly, when~\({ q < 1 }\)) better than in the low-frequency dominating scenario. For the actual solutions~\({ u }\) in \Cref{thm:existence-solitary}, we must, at least when~\({ q \geq 1 }\), have \({ \norm{ u_{ \textnormal{lo}} }_{ \infty } > \norm{ u_{ \textnormal{hi}} }_{ \infty } }\), because \({ \norm{ u }_{ \infty } \lesssim \norm{ u_{ \textnormal{hi}} }_{ \infty } \lesssim \mu^M }\) with \({ M = \infty }\) leads to the contradiction \({ u = 0 }\) in the high-frequency dominating case.
\end{remark}

\Cref{thm:wave-speed-lower-bound-solitary,thm:bound-Hs-norm-mu} now immediately imply the following result.
\begin{corollary} \label{thm:wavespeed-final-bound}
The estimate
\begin{equation*}
\nu - \mathfrak{m}(0) \gtrsim \mu^{q \alpha}
\end{equation*}
holds uniformly over the set of special near minimisers~\eqref{eq:special-near-min}.
\end{corollary}
Moreover, \({ A_P \to 1 }\) as~\({ P \to \infty }\) in the construction~\eqref{eq:special-min-seq} of~\({ \widetilde{ u }_k }\) from~\({ u_P^\star }\), so \Cref{thm:bound-Hs-norm-mu} also yields that \({ \norm{ u_P^\star }_{ \infty } \lesssim \mu^\alpha }\) uniformly in~\({ P \geq P_\mu }\) (possibly enlarged). But then, similarly as \Cref{thm:wave-speed-lower-bound-solitary}, we~get
\begin{equation*}
\nu_P - \mathfrak{m}(0) \gtrsim \mu^{q \alpha} + o\bigl(\norm{ u_P^\star }_{ \infty }^q\bigr) \gtrsim \mu^{q \alpha},
\end{equation*}
which leads to
\begin{equation*}
\nu_P - \mathfrak{m}(0) \eqsim \mu^{q \alpha} \eqsim \norm{ u_P^\star }_{ \infty }^q
\end{equation*}
with help of~\eqref{eq:wavespeed-est-from-equation}. This concludes the proof of~\Cref{thm:existence-periodic}.

\begin{lemma} \label{thm:special-near-min-nonlinearity-bnd}
Special near minimisers satisfy
\begin{equation*}
\mathcal{N}(u) \lesssim - \mu^{1 + q \alpha}, \qquad \mathcal{N}_q(u) \lesssim - \mu^{1 + q \alpha} \qquad \text{and} \qquad \mathcal{N}_\textnormal{r}(u) = o( \mu^{1 + q \alpha}).
\end{equation*}
\end{lemma}
\begin{proof}
Since~\({ - \mathcal{L}(u) \leq \mathfrak{m}(0) \mu }\), we find from~\eqref{eq:E-inf-estimate} that
\begin{equation*}
\mathcal{N}(u) = \mathcal{E}(u) - \mathcal{L}(u) \lesssim - \mu^{1 + q \alpha},
\end{equation*}
and
\begin{equation*}
\abs{ \mathcal{N}_\textnormal{r}(u) } = \int_{ \R } o\bigl(\abs{ u }^{2 + q}\bigr) \dee x = o\bigl( \mu \norm{ u }_{ \infty }^q \bigr) = o(\mu^{1 + q \alpha})
\end{equation*}
by~\Cref{thm:bound-Hs-norm-mu}.
\end{proof}

\begin{proposition} \label{thm:strict-subhomogeneity}
There exists~\({ \mu_\star > 0 }\) such that~\({ \mu \mapsto I_\mu }\) is strictly subhomogeneous on~\({ (0, \mu_\star) }\).
\end{proposition}
\begin{proof}
Fix~\({ a > 1 }\) and note that ~\({ \norm{ a^{ \frac{1}{2}} u }_{ s } \lesssim \mu^{ \frac{1}{2}} < R }\) for any special near-minimiser~\({ u }\). Estimating
\begin{align*}
I_{a \mu} \leq \mathcal{E}(a^{ \frac{1}{2}} u)  &= \mathcal{L}( a^{ \frac{1}{2}} u) + \mathcal{N}( a^{ \frac{1}{2}} u) \\
&= a \mathcal{L}( u) + a^{ \frac{1}{2}q} \mathcal{N}_q( u) + \mathcal{N}_\textnormal{r}( a^{ \frac{1}{2}} u) \\
&= a \mathcal{E}(u) + \bigl( a^{ \frac{1}{2}q} - a \bigr) \mathcal{N}_q(u) + \mathcal{N}_\textnormal{r}( a^{ \frac{1}{2}} u) - a \mathcal{N}_\textnormal{r} (u) \\
& \leq a \mathcal{E}(u) - c \bigl( a^{ \frac{1}{2}q} - a \bigr) \mu^{1 + q \alpha} + o(\mu^{1 + q \alpha}),
\end{align*}
where~\({ c > 0 }\), we may finally choose~\({ u = \widetilde{ u }_k }\) for the special minimising sequence~\({ \Set{ \widetilde{ u }_k }_k }\) and let~\({ k \to \infty }\). It follows that
\begin{equation*}
I_{a \mu} \leq a I_\mu - c \bigl( a^{ \frac{1}{2}q} - a \bigr) \mu^{1 + q \alpha} + o(\mu^{1 + q \alpha}) < I_\mu.
\end{equation*}
\end{proof}

\section{Concentration-compactness argument for solitary waves} \label{sec:concentration-compactness}

In this section we establish~\Cref{thm:existence-solitary} with help of Lions' concentration-compactness principle~\autocite[Lemma~III.1 and Remark~III.3]{Lio1984a}, stated in a suitable version below. Lions' principle, originally proved for~\({ \textnormal{H}^s }\) with~\({ s \in \N }\), generalises also to the fractional setting. Specifically, this concerns property~\labelcref{thm:cc-dichotomy-seminorm} under \enquote{dichotomy}, where we refer to~\autocite[Proposition~3.1 and Corollary~3.2]{ParSal2019a} for a derivation when~\({ s \in (0, 1) }\)---which together with Lions' result extends to all~\({ s > 0 }\).

\begin{theorem}[Concentration-compactness principle]
\label{thm:cc}
Every bounded sequence~\({\{\eta_k\}_{k \in \N} }\) in~\({ \textnormal{H}^s }\) satisfying
\begin{equation*}
\norm{\eta_k}_{0}^2 \xrightarrow[k \to \infty]{} \lambda > 0
\end{equation*}
admits a subsequence, still denoted by~\({\{\eta_k\}_k}\), for which one of the following phenomena takes place:
\begin{description}[leftmargin=\parindent,itemindent=.4em,font=\normalfont\itshape\color{artcolor}]
\item[Concentration:]
There exists a sequence~\({\{x_k\}_k \subset \R}\) such that
\begin{equation*}
\inf_{k \in \N} \int_{B_r(x_k)}\!\!\!\!\!\!\!\!\!\!\!\! \abs{\eta_k}^2 \dee x \xrightarrow[r \to \infty]{} \lambda.
\end{equation*}
\item[Vanishing:]
For all~\({r > 0}\) it is true that
\begin{equation*}
\sup_{y \in \R} \int_{B_r(y)}\!\!\!\!\!\!\!\!\!\! \abs{\eta_k}^2 \dee x \xrightarrow[k \to \infty]{}  0.
\end{equation*}
\item[Dichotomy:]
There exist a value~\({\theta \in (0, \lambda)}\), a sequence~\({ \Set{ x_k }_k \subset \R }\) and bounded sequences~\({\Set[\big]{\eta_k^{(1)}}\vphantom{ \eta}_k^{}}\), \({ \Set[\big]{\eta_k^{(2)}}\vphantom{ \eta}_k^{} }\) in~\({ \textnormal{H}^s }\), such that
\begin{enumerate}[itemsep=1em,itemindent=.5\parindent,ref=\roman*)]
\item \label{thm:cc-dichotomy-convergence}
\({\norm[\big]{\eta_k^{} - \eta_k^{(1)} - \eta_k^{(2)}}_{0} \to  0, \quad \norm[\big]{\eta_k^{(1)}}_{0}^2 \to \theta, \quad \text{and} \quad \norm[\big]{\eta_k^{(2)}}_{0}^2 \to \lambda - \theta}\);
\item \label{thm:cc-dichotomy-support}%
\({%
\begin{aligned}[t]
\support \eta_k^{(1)} = \Set*{ \abs{ x - x_k } \leq A_k } \\
\support \eta_k^{(2)} = \Set*{ \abs{ x - x_k } \geq B_k }
\end{aligned}%
\mathrel{\raisebox{-.9em}{%
\quad\text{for\quad\({ A_k, B_k \to \infty }\)\quad satisfying\quad\({ \displaystyle\frac{A_k}{B_k} \to 0 }\); and}%
}}%
}\)
\item \label{thm:cc-dichotomy-seminorm}
\({\displaystyle\liminf_k \left( \bigl[ \eta_k \bigr]_s^2 - \bigl[ \eta_k^{(1)} \bigr]_s^2 - \bigl[ \eta_k^{(2)} \bigr]_s^2  \right) \geq 0}\), where~\({[\cdot]_s^2 \coloneqq \norm{}_{s}^{2} - \norm{}_{0}^{2}}\) is a seminorm.
\end{enumerate}
\end{description}
\end{theorem}
Practically, we may rescale and assume that for all~\({ k }\),
\begin{equation*}
\norm[\big]{\eta_k}_{0}^2 = \lambda, \qquad \norm[\big]{\eta_k^{(1)}}_{0}^2 = \theta, \qquad \text{and} \qquad \norm[\big]{\eta_k^{(2)}}_{0}^2 = \lambda - \theta.
\end{equation*}

We apply~\Cref{thm:cc} to the special minimising sequence~\({ \Set{ \widetilde{ u }_k }_k }\) for~\({ \mathcal{E} }\) over~\({ U_\mu^s }\) from~\cref{sec:special-min-seq}, dropping the tilde in~\({ \widetilde{ u }_k }\) for clarity. Note that we may always assume that~\({ u_k }\) is at least in~\({ U_\mu^1 }\), because we may let~\({ u_k }\) be constructed from the periodic minimisers corresponding to~\({ s = 1 }\), which is \textit{a~priori} best for Lipschitz~nonlinearities.

\begin{lemma} \label{thm:concentrates}
Let~\({\tilde{s} \in \left(0, s\right)}\) and suppose that a subsequence of~\({\Set{u_k}_k}\) \enquote{concentrates}. Then a subsequence of~\({\left\{ u_k(\cdot + x_k) \right\}_k}\) converges in~\({\textnormal{H}^{\tilde{s}}}\) to a minimiser of~\({\mathcal{E}}\) over~\({U_\mu^s}\).
\end{lemma}
\begin{proof}
Let~\({\epsilon > 0}\) and define~\({v_k \coloneqq u_k(\cdot + x_k)}\), so that by assumption
\begin{equation*}
\int_{\abs{x} > r} \!\!\!\!\!\!\!\!\!\! v_k^2 \dee x < \epsilon
\end{equation*}
for all sufficiently large~\({r > 0}\), uniformly in~\({k}\). Since \({\{v_k\}_k \subset U_\mu^s}\) is bounded in~\({\textnormal{H}^{s}}\), it converges weakly---up to a subsequence---in~\({\textnormal{H}^{s}}\) to some~\({v \in U_\mu^s}\). Moreover, boundedness implies \({\textnormal{L}^2}\)-concentration of the frequency spectrum, because
\begin{equation*}
\int_{\abs{\xi} > r'} \!\!\!\!\!\!\!\!\!\! \abs{\widehat{v}_k}^2 \dee \xi \leq \langle r' \rangle^{-2s} \norm{v_k}_{s}^2 < \epsilon
\end{equation*}
for sufficiently large~\({r' > 0}\), uniformly in~\({k}\). This in turn yields equicontinuity in~\({\textnormal{L}^2}\) by estimating
\begin{equation*}
\int_{\R} \abs*{v_k(\cdot + y) - v_k}^{2} \dee x = \int_{\R} \abs*{\left( \textnormal{e}^{\textnormal{i} y \xi}  - 1 \right) \widehat{v}_k(\xi)}^{2}  \dee \xi \lesssim \abs{y}^{2} \int_{ \abs{\xi} \leq r'} \!\!\!\!\!\!\!\!\!\! \abs{\widehat{v}_k}^{2}  \dee \xi + \int_{\abs{\xi} > r'} \!\!\!\!\!\!\!\!\!\! \abs{\widehat{v}_k}^2 \dee \xi < 2\epsilon,
\end{equation*}
valid uniformly for all sufficiently small~\({y}\) and uniformly in~\({k}\). Kolmogorov--Riesz--Sudakov's compactness theorem then shows that~\({\{v_k\}_k}\) converges, up to a subsequence, in~\({\textnormal{L}^2}\), with limit which must be~\({v}\). Interpolating
\begin{equation*}
\norm{w}_{\tilde{s}}^{} \leq \norm{w}_{0}^{1 - ( \tilde{ s } / s)}  \norm{w}_{s}^{\tilde{ s }/s} \lesssim \norm{w}_{0}^{1 - (\tilde{ s } / s)}
\end{equation*}
with~\({w \coloneqq v_k - v}\) for clarity, upgrades convergence to~\({\textnormal{H}^{\tilde{ s }}}\), and by continuity of~\({\mathcal{E}}\) we are done.
\end{proof}

It remains to exclude vanishing and dichotomy. Note that there is an easily corrected flaw in the proof of vanishing in~\autocite[Lemma~5.2]{EGW2012} (the fourth inequality); for example, one may use the Gagliardo--Nirenberg inequality as in the proof of~\Cref{thm:vanishing} below, or apply Hölder's inequality together with~\({ \textnormal{H}^1 \hookrightarrow \textnormal{L}^\infty }\).
\begin{lemma} \label{thm:vanishing}
Vanishing does not occur.
\end{lemma}
\vspace*{-1em} 
\enlargethispage{.5\baselineskip} 
\begin{proof}
Seeking to contradict~\Cref{thm:special-near-min-nonlinearity-bnd}, we first observe that
\begin{equation*}
\abs{\calN(u_k)} \lesssim \norm{u_k}_{\textnormal{L}^{2 + q}}^{2 + q} \eqsim \sum_{j \in \Z} \norm{u_{k, j}}_{\textnormal{L}^{2 + q}}^{2 + q},
\end{equation*}
where~\({u_{k, j} \coloneqq u_k \varphi_j}\) and~\({\Set{\varphi_j}_j}\) is a smooth partition of unity with~\({\varphi_j(x) \equiv 1}\) for~\({ \abs{x - j} \leq \tfrac{1}{4}}\) and \({\support \varphi_j = \bigl[j - \frac{3}{4}, j + \frac{3}{4}\bigr]}\). Let~\({v}\) equal any~\({u_{k, j}}\). Estimating
\begin{equation*}
\norm{v}_{\textnormal{L}^{2 + q}}^{2 + q} \lesssim \norm{v}_{s}^{q / 2s} \norm{v}_{0}^{2 + q - (q/2s)} \leq \norm{v}_{s}^{2} \norm{v}_{0}^{q},
\end{equation*}
by the Gagliardo--Nirenberg inequality, valid since~\({2s > q / (2 + q)}\) always holds for the chosen special minimising sequence, it then follows that
\begin{equation*}
\abs{\calN(u_k)} \lesssim \Biggl( \sup_{j \in \Z} \int_{\R} \abs{u_{k, j}}^2 \dee x \Biggr)^{q/2} \underbrace{\sum_{j \in \Z} \norm{u_{k, j}}_{s}^{2}}_{\eqsim \norm{u_k}_{s}^2 < R^2} \xrightarrow[k \to \infty]{}  0
\end{equation*}
if~\({\Set{u_k}_k}\) vanishes, which is absurd.
\end{proof}

Suppose now that dichotomy occurs, so that~\({ \Set{ u_k }_k }\) admits decomposing sequences~\({\Set[\big]{u_k^{(1)}}_k}\),~\({\Set[\big]{u_k^{(2)}}_k}\), with
\begin{equation} \label{eq:dichotomy-splitting}
u_k^{(1)} \in U_{\frac{\theta}{2}}^s, \qquad u_k^{(2)} \in U_{\mu - \frac{\theta}{2}}^s \qquad \text{and} \qquad u_k^{(1)} + u_k^{(2)} \in U_\mu^s \qquad \text{for all~\({k}\);}
\end{equation}
see the proof of~\Cref{thm:dichotomy}. If separation of~\({u_k^{(1)}}\) and~\({u_k^{(2)}}\) leads to the energetic decomposition
\begin{equation}
\label{eq:energetic-decomp}
\lim_{k} \left[ \calE \bigl( u_k^{(1)} +  u_k^{(2)} \bigr) - \calE \bigl( u_k^{(1)} \bigr) - \calE \bigl( u_k^{(2)} \bigr) \right] = 0,
\end{equation}
then subsequently
\begin{equation*}
\lim_{k} \left[ \calE \bigl( u_k^{(1)} \bigr) + \calE \bigl( u_k^{(2)} \bigr) \right] = \lim_{k} \calE(u_k) = I_\mu,
\end{equation*}
using that
\begin{equation*}
\abs[\big]{\calE(u_k) - \calE \bigl( u_k^{(1)} +  u_k^{(2)} \bigr)} \leq \sup_{u \in U_\mu^s} \norm{\calE'(u)}_{0}  \norm[\big]{u_k^{} - u_k^{(1)} - u_k^{(2)}}_{0} \to  0
\end{equation*}
from property~\labelcref{thm:cc-dichotomy-convergence} and boundedness of~\({\norm{\calE'(u)}_{0}}\) on~\({U_\mu^s}\). In light of strict subadditivity of~\({\mu \mapsto I_\mu}\), we~then get the contradiction
\begin{equation*}
I_\mu < I_\frac{\theta}{2} + I_{\mu - \frac{\theta}{2}} \leq \lim_{k} \left[ \calE \bigl( u_k^{(1)} \bigr) + \calE \bigl( u_k^{(2)} \bigr) \right] = I_\mu.
\end{equation*}

Accordingly, it suffices to establish~\eqref{eq:energetic-decomp}. And to this end, note that since~\({\calN}\)~is a local operator, it eventually splits as
\begin{equation*}
\calN \bigl( u_k^{(1)} +  u_k^{(2)} \bigr) = \calN \bigl( u_k^{(1)} \bigr) + \calN \bigl( u_k^{(2)} \bigr),
\end{equation*}
whereas~\({\calL}\) satisfies
\begin{equation*}
\calL \bigl( u_k^{(1)} +  u_k^{(2)} \bigr) = \calL \bigl( u_k^{(1)} \bigr) + \calL \bigl( u_k^{(2)} \bigr) - \innerproduct[\big]{L u_k^{(1)}}{u_k^{(2)}}_{0}.
\end{equation*}
In order to show that the nonlocal interaction disappears as~\({ k \to \infty }\), one can introduce certain commutators and prove that their operator norms vanish~\autocite{Mae2019a}. Based on uniform continuity of~\({ \xi \mapsto m(\xi) / \langle \xi \rangle^{s} }\), which holds automatically in our case, this is applicable for a large class of~symbols. For~convolution operators, however, it seems more enlightening to work directly on the \enquote{physical side}, assuming just integrability of the~kernel.

\begin{lemma} \label{thm:convolution-dichotomy}
Let~\({ f \in \textnormal{L}^1 }\) and~\({ \Set{ v_k }_k, \Set{ w_k }_k \subset \textnormal{L}^2 }\) be bounded and satisfy
\begin{equation*}
\support v_k = \Set*{ \abs{ x } \leq  A_k } \qquad \text{and} \qquad \support w_k = \Set*{ \abs{ x } \geq B_k }
\end{equation*}
for~\({ 0 \leq A_k, B_k \xrightarrow[ k \to \infty ]{  } \infty}\) with~\({ B_k - A_k \to \infty }\). Then~\({ \innerproduct{ f \ast v_k }{ w_k }_0 \xrightarrow[ k \to \infty ]{  } 0 }\).
\end{lemma}
\begin{proof}
An inspection of the proof of Young's inequality~\autocite[20.3.2~Proposition]{GasquetWitomski1998} shows that
\begin{equation*}
\abs*{\innerproduct{ f \ast v_k }{ w_k }_0}^2 \leq \norm{ f \ast v_k }_{ \textnormal{L}^2(\support w_k) }^2 \norm{ w_k }_{ 0 }^2  \leq \norm{ f }_{ \textnormal{L}^1 } \norm{ w_k }_{ 0 }^2  \!\!\!\int\displaylimits_{ \support w_k }  \!\!\!\!\!\!\!\!\!\!\!\!\!\!\!\!  \int\displaylimits_{ \qquad\;\;\; \support v_k } \!\!\!\!\!\!\!\!\!\!\!\! \abs{ f(x - y) } \, \abs{ v_k(y) }^{ 2 } \dee y \dee x
\end{equation*}
with help of the Cauchy--Schwarz inequality. Changing the order of integration then yields
\begin{equation*}
\int\limits_{\mathclap{  \Set*{ \abs{ x } \geq B_k } }} \int_{- A_k}^{A_k} \! \abs{ f(x - y) } \, \abs{ v_k(y) }^{ 2 } \dee y \dee x = \int_{ -A_k }^{ A_k } \! \abs{ v_k(y) }^{ 2 } \int\limits_{\mathclap{ \Set*{ \abs{ x + y } \geq B_k } } } \abs{ f(x) } \dee x \dee y,
\end{equation*}
and so, since~\({ \Set*{ \abs{ x + y } \geq B_k } \subseteq \Set*{ \abs{ x } \geq B_k - A_k } }\) for all~\({ y \in [-A_k, A_k] }\), we end up with
\begin{equation*}
\abs*{\innerproduct{ f \ast v_k }{ w_k }_0}^2 \leq \norm{ f }_{ \textnormal{L}^1} \norm{ v_k }_{ 0 }^2 \norm{ w_k }_{ 0 }^2 \int\limits_{\mathclap{  \Set*{ \abs{ x } \geq B_k - A_k } } } \abs{ f(x) } \dee x \xrightarrow[ k \to \infty ]{  } 0.
\end{equation*}
\end{proof}
\begin{corollary} \label{thm:dichotomy}
Dichotomy does not occur when~\({ \mu \in (0, \mu_\star) }\), with~\({ \mu_\star }\) as in~\Cref{thm:strict-subhomogeneity}.
\end{corollary}
\begin{proof}
Contrariwise, assume the existence of decomposing sequences~\({\Set[\big]{u_k^{(1)}}\vphantom{u}_k^{}}\) and~\({\Set[\big]{u_k^{(2)}}\vphantom{u}_k^{}}\) from \Cref{thm:cc}, rescaled to satisfy~\({\norm[\big]{u_k^{(1)}}_{0}^2 = \theta}\) and~\({\norm[\big]{u_k^{(2)}}_{0}^2 = 2\mu - \theta}\) for all~\({k}\). Flipping signs in property~\labelcref{thm:cc-dichotomy-seminorm} shows that
\begin{equation*}
\limsup\nolimits_k^{} \,\bigl[ u_k^{(1)} + u_k^{(2)} \bigr]_s^2 \leq \limsup\nolimits_k^{} \, \bigl[ u_k^{} \bigr]_s^2
\end{equation*}
with help of the triangle inequality, which in combination with property~\labelcref{thm:cc-dichotomy-convergence} give
\begin{equation*}
\limsup\nolimits_k^{} \norm[\big]{u_k^{(1)} + u_k^{(2)}}_{s} \leq \limsup\nolimits_k^{} \norm[\big]{u_k^{}}_{s} < R.
\end{equation*}
Since~\({u_k^{(1)}}\) and~\({u_k^{(2)}}\) eventually separate (property~\labelcref{thm:cc-dichotomy-support}), we also obtain
\begin{equation*}
\limsup\nolimits_k^{} \norm[\big]{u_k^{(j)}}_{s} \leq \limsup\nolimits_k^{} \norm[\big]{u_k^{(1)} + u_k^{(2)}}_{s}, \qquad j = 1, 2,
\end{equation*}
and so, without loss of generality, we may assume~\eqref{eq:dichotomy-splitting}.

Following the discussion prior to~\Cref{thm:convolution-dichotomy}, it remains to show that~\({ \innerproduct[\big]{L u_k^{(1)}}{u_k^{(2)}}_{0} \xrightarrow[ k \to \infty ]{  } 0 }\). But~this is immediate from~\Cref{thm:convolution-dichotomy} and~property~\labelcref{thm:cc-dichotomy-support} after spatial translations~\({ x \mapsto x - x_k }\).
\end{proof}

\enlargethispage{-\baselineskip} 

We conclude from~\Cref{thm:concentrates,thm:vanishing} and \Cref{thm:dichotomy} that~\({ \mathcal{E} }\) has a minimiser over~\({ U_\mu^s }\). Combined with the estimates in~\Cref{thm:wave-speed-lower-bound-solitary,thm:bound-Hs-norm-mu} and \Cref{thm:wavespeed-final-bound}, we deduce, similarly as in the periodic case, that
\begin{equation*}
\nu - \mathfrak{m}(0) \eqsim \mu^{q \alpha} \eqsim \norm{ u }_{ \infty }^q.
\end{equation*}
This completes the proof of~\Cref{thm:existence-solitary}.

\section{Additional features} \label{sec:additional}

As a consequence of the analysis in the proof of \Cref{thm:bound-Hs-norm-mu}, we obtain a nonexistence result for small solitary waves in \({ \textnormal{H}^s \cap \textnormal{L}^\infty }\) when the nonlinearity is too strong, which demonstrates the optimality of \({ q < 4 \ell }\) in Assumptions~\ref{assumption:dispersion} and~\ref{assumption:nonlinearity}. 

\begin{theorem}[Nonexistence] \label{thm:nonexistence}
Let~\({ s > 0 }\) be as in~\eqref{eq:sobolev-index}. If~\({ q \geq 4\ell }\) in Assumptions~\ref{assumption:dispersion} and~\ref{assumption:nonlinearity}, then there are no nonzero solutions \({ u \in \textnormal{H}^s \cap \textnormal{L}^\infty }\) of equation~\eqref{eq:traveling-wave} with speed~\({\!\nu }\) satisfying~\({ \!\nu - \mathfrak{m}(0) \gtrsim -\norm{ u }_{ \infty }^{ q }  }\) provided~\({ \norm{ u }_{ s }  }\) and~\({ \norm{ u }_{ \infty }  }\) are sufficiently small. In~particular, this excludes small solitary waves in \({ \textnormal{H}^s \cap \textnormal{L}^\infty }\) with supercritical speed when~\({ q \geq 4 \ell }\).
\end{theorem}
\begin{proof}
We split~\({ u }\) into~\({ u_{ \textnormal{lo}} }\) and~\({ u_{ \textnormal{hi}} }\) exactly as in~\eqref{eq:freq-decomp}, so that~\eqref{eq:low-frequency-comp}–\eqref{eq:high-frequency-comp} hold with~\({ \mathcal{E}'(u) + \nu \mathcal{Q}'(u) \equiv 0 }\). Closely following the proof of~\Cref{thm:bound-Hs-norm-mu}, suppose first that~\({ \norm{ u_{ \textnormal{hi}} }_{ \infty } \geq \norm{ u_{ \textnormal{lo}} }_{ \infty }  }\). Without repeating the calculations we then obtain from~\eqref{eq:high-freq-comp-geometric} that
\begin{equation*}
\norm{ u_{ \textnormal{hi}} }_{ \infty } \lesssim \norm{ u }_{ s } \norm{ u_{ \textnormal{hi}} }_{ \infty }^{ q },
\end{equation*}
provided~\({ \norm{ u }_{ \infty }  }\) is sufficiently small. Since~\({ q \geq 4\ell \geq 1 }\) in this scenario, we deduce that~\({ u = 0 }\) if~\({ \norm{ u }_{ s }  }\) is sufficiently small.

Suppose instead that~\({ \norm{ u_{ \textnormal{lo}} }_{ \infty } > \norm{ u_{ \textnormal{hi}} }_{ \infty }  }\). Due to~\({ \!\nu - \mathfrak{m}(0) \gtrsim - \norm{ u }_{ \infty }^{ q }  }\), estimate~\eqref{eq:wavespeed-nearzero-inf-bnd} now becomes
\begin{equation*}
\nu - \mathfrak{m}(\xi) + \frac{ c \mathfrak{m}^{(2\ell)}(0)}{(2 \ell)!} \xi^{2 \ell} \gtrsim - \norm{ u }_{ \infty }^q
\end{equation*}
for some~\({ c > 0 }\) when~\({ \abs{ \xi } < \xi_0 }\). By redoing estimates~\eqref{eq:optimal-inf-estimate1}–\eqref{eq:optimal-inf-estimate4} with the appropriate modifications, one obtains
\begin{equation*}
\norm{ u_{ \textnormal{lo}}^{(2\ell)} }_{ 0 } \lesssim \norm{ u }_{ 0 } \norm{ u }_{ \infty }^{ q } \leq  \norm{ u }_{ s } \norm{ u }_{ \infty }^{ q } 
\end{equation*}
for sufficiently small~\({ \norm{ u }_{ \infty }  }\), which implies that
\begin{equation*}
\norm{ u }_{ \infty } \lesssim \norm{ u_{ \textnormal{lo}} }_{ \infty } \lesssim \norm{ u_{ \textnormal{lo}} }_{ 0 }^{1 - \frac{1}{4\ell}} \norm{ u_{ \textnormal{lo}}^{(2\ell)} }_{ 0 }^{ \frac{1}{4\ell} } \lesssim \norm{ u }_{ s }  \norm{ u }_{ \infty }^{\frac{q}{4\ell}}
\end{equation*}
by the Gagliardo--Nirenberg inequality. If~\({ \norm{ u }_{ \infty } \leq 1 }\), then for sufficiently small~\({ \norm{ u }_{ s }  }\) we conclude that \({ u = 0 }\) is the only possibility when~\({ q \geq 4\ell }\).
\end{proof}

We finally establish with a basic argument that bounded solutions of~\eqref{eq:traveling-wave} with supercritical speed are either waves of elevation or waves of depression in the special case when the convolution kernel~\({ K }\) is nonnegative. This result is already known for the Whitham equation~\autocite[Corollary~4.4]{EhrWah2019a}.

\begin{theorem}[Sign of wave profile] \label{thm:wave-sign}
Suppose~\({ K }\) is nonnegative and let~\({ u \neq 0 }\) be a~bounded solution of~\eqref{eq:traveling-wave} with supercritical wave speed~\({\!\nu > \mathfrak{m}(0) }\). If~\({ n }\) is homogeneous, then~\({ u }\) has a one-sided profile with~\({ \sign u = \sign \gamma }\) almost everywhere, where~\({ \gamma }\) is as in Assumption~\ref{assumption:nonlinearity}. The~same conclusion also holds for inhomogeneous~\({ n }\) when~\({ \norm{ u }_{ \infty }  }\) is sufficiently~small.
\end{theorem}
\begin{proof}
It suffices to consider~\({ n_q(u) = \gamma \abs{ u }^{1 + q} }\), as the sign-dependent case~\({ n_q(u) = \gamma u \abs{ u }^q }\) follows from~\({ u \mapsto -u }\) and arguing with the (essential) supremum of~\({ u }\) instead of the infimum.

If~\({ \gamma > 0 }\), suppose that~\({ u_\ast \coloneqq \operatorname{ess\, inf} u < 0 }\). Let~\({ \epsilon > 0 }\) and---being slightly informal---let~\({ x_\epsilon }\) be any point such that~\({ u(x_{ \epsilon}) < u_\ast + \epsilon }\). We~find that~\({ Lu(x_{ \epsilon}) \geq L(u_\ast ) = \widehat{  K  }(0) \, u_\ast =  \mathfrak{m}(0) \, u_\ast  }\) because~\({ K \geq 0 }\), and~so
\begin{equation} \label{eq:wave-sign-epsilon}
n(u(x_\epsilon)) = \nu u(x_\epsilon) - Lu(x_\epsilon) \leq \left( \nu - \mathfrak{m}(0) \right) u_\ast + \epsilon \nu.
\end{equation}
Since~\({ \!\nu > \mathfrak{m}(0) }\), the right-hand side in~\eqref{eq:wave-sign-epsilon} becomes negative for~\({ \epsilon }\) sufficiently small. This is a contradiction if~\({ n = n_q }\), because~\({ n_q(u(x_\epsilon)) > 0 }\), and also in the inhomogeneous case provided~\({ \norm{ u }_{ \infty }  }\) is sufficiently~small.

If~\({ \gamma < 0 }\), one may argue analogously with~\({ \operatorname{ess\, sup} u }\).
\end{proof}

\section*{Acknowledgements}

The author acknowledges the support by research grant no.~250070 from The Research Council of Norway. Valuable suggestions from two anonymous referees that helped to improve the paper are gratefully acknowledged.

\appendix

\section{Sufficient conditions for symbols to be in the Wiener class~\({ \textnormal{W}_0 }\)} \label{app:wienerclass-sufficientconditions}

Sufficient conditions for symmetric symbols~\({ \mathfrak{m} }\) with weak decay to be in the Wiener class~\({ \textnormal{W}_0 }\) of functions with absolutely integrable inverse Fourier transform are for instance
\begin{itemize}[itemsep=0em]
\item
\({ \mathfrak{m} \in \textnormal{AC}_{ \textnormal{loc}} }\) satisfying~\({ \abs{ \mathfrak{m} (\xi)} \lesssim \langle \xi \rangle^{\sigma} }\) and~\({ \abs{ \mathfrak{m}'(\xi)} \lesssim \langle \xi \rangle^{\sigma'} }\) almost everywhere for~\({ \sigma < 0 }\) and~\({ \sigma' \in \R }\) with \({ \sigma + \sigma' < -1 }\); see~\autocite[Theorem~1]{LifTri2010a} and~\autocite[Corollary~2.2]{LifTri2011a}. This~directly extends the \({ \textnormal{S}_\infty^\sigma }\)~case. Here~\({ \textnormal{AC}_{ \textnormal{loc}} }\) is the space of locally absolutely continuous functions;
\item
\({ \mathfrak{m} \in \textnormal{AC}_{ \textnormal{loc}} }\) satisfying~\({ \mathfrak{m} \in \textnormal{L}^{p_1} }\) and~\({ \mathfrak{m}' \in \textnormal{L}^{p_2} }\) for~\({ 1 \leq p_1 < \infty }\), \({ 1 < p_2 < \infty }\) fulfilling \({ \frac{1}{p_1} + \frac{1}{p_2} > 1 }\) \autocite[Theorem~1.1]{Lif2010a}; and
\item
\({ \mathfrak{m}  }\) being \textit{quasi-convex} on~\({ (0, \infty) }\), meaning that~\({ \mathfrak{m} \in \textnormal{AC}_{ \textnormal{loc}} }\) with~\({ \mathfrak{m}' }\) locally of bounded variation and~\({ \int_{ 0 }^{ \infty } \xi  \, \abs{\dee \mathfrak{m}'(\xi)} < \infty }\) (Riemann--Stieltjes integral). Example: \({ \mathfrak{m}(\xi) = \left( 1 + \log( 1+ \abs{ \xi }) \right)^{- \alpha} }\), for any~\({ \alpha > 0 }\); see~\autocite[Theorem~6.3.11]{ButNes1971a} and~\autocite[Theorem~5.4]{LifSamTri2012a}.
\end{itemize}

\phantomsection\addcontentsline{toc}{section}{References}
\printbibliography
\end{document}